\documentclass{article}
\usepackage[utf8]{inputenc}
\usepackage{amsmath}
\usepackage{amsfonts}
\usepackage{amsthm}
\usepackage{commath}
\usepackage{mathtools}
\usepackage{amssymb}
\usepackage{graphicx}
\usepackage{floatrow}
\usepackage{forest}
\graphicspath{ {./images/} }
\usepackage{float}
\usepackage{titling}
\usepackage{subcaption}
\usepackage[document]{ragged2e}
\usepackage[labelformat=empty]{subcaption}
\usepackage{tabularx}
\usepackage[backend=biber,style=alphabetic]{biblatex}
\usepackage{enumitem}
\usepackage{thmtools}
\usepackage{yfonts}
\newtheorem{theorem}{Theorem}[section]
\newtheorem{lemma}[theorem]{Lemma}
\newtheorem{corollary}[theorem]{Corollary}
\newtheorem*{theoremoneone}{Theorem 1.1}
\newtheorem*{theoremonetwo}{Theorem 1.2}
\newtheorem*{lemmasixone}{Lemma 6.1}
\declaretheoremstyle[bodyfont=\normalfont,spaceabove = \baselineskip, 
spacebelow = \baselineskip]{normalhead}
\declaretheorem[name=Definition,sibling=theorem,style=normalhead]{definition}
\declaretheorem[name=Definition A.1,numbered=no,style=normalhead]{definitionaone}
\title{The First and Second Derivatives of the $q$-Rationals}
\author{Justin Lasker}

\centering
\date{October 18, 2024}

\begin{document}

\begin{titlepage}
\maketitle
\thispagestyle{empty}
\end{titlepage}

\justifying

\section{\textbf{Introduction and the Main Theorem}}
A \textit{quantum rational number}, or $q$\textit{-rational}, $[r]_q$ is a rational function on $q$ which depends on the rational number $r$. Introduced in 2019 by Sophie Morier-Genoud and Valentin Ovsienko \cite{Morier_Genoud_2019}, the quantum rational numbers generalize the concept of quantum integers, or $q$-integers, as formulated by Gauss.
\newline~
\newline
Gauss's original concept of the $q$-integer is a function $\left[n\right]_{q}$ that takes on the value $n$ when $q=1$. They are classically defined by the equation
\begin{equation*}
\left[n\right]_{q}\coloneqq\frac{1-q^n}{1-q}=1+q+q^2+\cdots+q^{n-1}
\end{equation*}
Similarly, the $q$-rational is a function $[a/b]_{q}$ that takes on the value $a/b$ when $q=1$. I will now illustrate Morier-Genoud and Ovsienko's concept of the $q$-rationals through the linear fractional transformation and the quantization map. The linear fractional transformation
\begin{align*}
\begin{pmatrix}
a & c \\
b & d
\end{pmatrix}(z) &\mapsto \frac{az+c}{bz+d} \\
\text{PSL}(2,\mathbb{Z}) &\mapsto \mathbb{Q}\cup\infty
\end{align*}
represents $\text{SL}(2,\mathbb{Z})$-action on the real projective line $\mathbb{Q}\cup\infty$. The group of linear fractional transformations is isomorphic to the projective special linear group $\text{PSL}(2,\mathbb{Z})$, which is generated by the matrices
\[T = \begin{pmatrix}
1 & 1 \\
0 & 1
\end{pmatrix}\text{ and }S = \begin{pmatrix}
0 & -1 \\
1 & 0
\end{pmatrix}\]
In turn, we can use the \textit{$q$-deformed} matrices
\[T_q\begin{pmatrix}
q & 1 \\ 
0 & 1
\end{pmatrix}\text{ and }S_q = \begin{pmatrix}
0 & -q^{-1} \\
1 & 0
\end{pmatrix}\]
to generate the group $\text{PSL}_q(2,\mathbb{Z})$; this group is isomorphic to the group of $\text{SL}_q(2,\mathbb{Z})$-actions on $\mathbb{Q}[q]\cup\infty$ \cite{leclere_morier_genoud}. The $q$-deformation of the rational numbers yields the $q$-rationals, which satisfy the following properties \cite{leclere_morier_genoud}:
\begin{equation*}
\begin{gathered}
[x+1]_q = q[x]_q+1 \\
\left[-\frac{1}{x}\right]_q=-q^{-1}\frac{1}{[x]_q}
\end{gathered}
\end{equation*}
in which I denote the $q$-deformation of $x \in \mathbb{Q}$ as $[x]_q$, a notation which I will use throughout the rest of the paper. The $q$-deformation
\begin{alignat*}{2}
[{}\cdot{}]_q : & \quad & M & \mapsto [M]_q\\
&\quad&\text{PSL}(2,\mathbb{Z}) &\mapsto \text{PSL}_q(2,\mathbb{Z})
\end{alignat*}
such that $[T]_q = T_q$, $[S]_q = S_q$, and $[AB]_q = [A]_q[B]_q$, commutes with the $SL_q(2, \mathbb{Z})$-action on $\mathbb{Q}[q]\cup\infty$ \cite{leclere_morier_genoud}:
\[
\left[\begin{pmatrix}
a & c \\
b & d
\end{pmatrix}(x)\right]_q = \left[\begin{pmatrix}
a & c \\
b & d
\end{pmatrix}\right]_q[x]_q\]
To date, the quantum rational numbers have no known closed-form expression. One strategy for finding such an expression is to calculate the Taylor series of the $q$-rationals at $q = 0$. Morier-Genoud and Ovsienko \cite{Morier_Genoud_2020} give examples of such series for several positive $q$-rationals and $q$\textit{-irrationals}, the latter of which are calculated via limits of $q$-rationals. (Negative $q$-rationals can only be expressed through Laurent series at $q = 0$ \cite{Morier_Genoud_2020}.) However, there has been little to no research to date on whether the $q$-rationals could be expressed through a Taylor series at some $q \neq 0$. The point $q = 1$, where the $q$-rationals equal the rationals, is a natural candidate. 
\newline~
\newline
In this paper, I prove original closed-form expressions for the first and second derivatives of the $q$-rationals at $q=1$, and demonstrate how my proofs share a structure that can be used to prove expressions for derivatives at higher orders as well. The expressions I found are shown below.
\begin{theorem}
Let $x$ be a rational number. The first derivative of its $q$-deformation $[x]_q$ at $q = 1$ is given by
\[
\frac{d}{dq}\left[x\right]_q\Bigr|_{q=1} = \frac{1}{2}\left(x^2 - x + 1 - f(x)^2\right)
\]
in which $f(x)$ is Thomae's function, namely
\[f(x) = \begin{dcases}
      \frac{1}{b} & x = \frac{a}{b},\mathrm{gcd}(a,b)=1 \\
      0 & x \notin \mathbb{Q}
   \end{dcases}
\]
\end{theorem}
\begin{theorem}
Let $a/b$ be an irreducible rational number. The second derivative of its $q$-deformation $[a/b]_q$ at $q = 1$ is given by
\[
\frac{d^2}{dq^2}\left[\frac{a}{b}\right]_q\Bigr|_{q=1} = F\left(\frac{a}{b}\right) + f\left(\frac{a}{b}\right)^2G\left(\frac{a}{b}\right)+20\sum_{n=1}^{b-1}{\left\langle\frac{n}{a}\right\rangle_ b H\left(\frac{n}{b}\right)}
\]
in which
\begin{itemize}
\item $f(x)$ is Thomae's function (as given in Theorem 1.1)
\item \(F(x) = \frac{1}{3}x^3-x^2+\frac{5}{3}x-1\)
\item \(G(x) = 1 - x\)
\item \(H(x) = x^2(1 - x)\)
\item \(\left\langle\frac{n}{a}\right\rangle_b \coloneqq \frac{na^{-1}\:\mathrm{mod}\:b}{b}-\frac{1}{2}\)
\newline
where $a^{-1}$ is the inverse of $a$ in the multiplicative group \(\left(\mathbb{Z}\slash b\mathbb{Z}\right)^\times\) and \(x\:\mathrm{mod}\:b\) is the integer between $0$ and $b-1$ corresponding to the congruence class of $x$ modulo $b$.
\end{itemize}
\end{theorem}
\noindent These results expand upon that of Adam Sikora \cite{sikora2020tangle}, who found a closed form expression for the first derivative of the $q$-rationals at $q=1$ in certain special cases. Hopefully, a completed Taylor series for the quantum rationals will ease their calculation for future researchers.
\newline~
\newline
The rest of the paper will be structured as follows:
\newline~
\newline
\textbf{Section 2:} Properties of the $q$-rationals.
\newline~
\newline
\textbf{Section 3:} Computer experimentations I used to find the two main results of the paper.
\newline~
\newline
\textbf{Section 4:} A schema for proving identities for arbitrary-order derivatives of the $q$-rationals.
\newline~
\newline
\textbf{Section 5:} Proof of my first derivative identity.
\newline~
\newline
\textbf{Section 6:} Proof of my second derivative identity.
\newline~
\newline
\textbf{Section 7:} The conclusion, which outlines possibilities for future research.
\newline~
\newline
In addition, the paper contains two appendices:
\newline~
\newline
\textbf{Appendix A:} Given the properties observed in Section 2, I offer an original proof that the definition of $q$-rationals based on continued fractions and that based on algebraic fractional transformations are equivalent. (This result was formerly proven by Morier-Genoud and Ovsienko \cite{Morier_Genoud_2020}.)
\newline~
\newline
\textbf{Appendix B:} The proof of a lemma that supplements Section 6.
\section{\normalfont \textbf{Properties} \(\mid\) The $q$-Rationals}
In this section, I will give two identities for the $q$-rational numbers, which I prove are in fact equivalent in Appendix A. First, I provide a continued fraction expansion definition: for
\[
\frac{a}{b}=a_0+\cfrac{1}{a_1+\cfrac{1}{a_2+\cfrac{1}{a_3+\cfrac{1}{\cdots}}}}\]
we have
\begin{equation*}
\left[\frac{a}{b}\right]_q\coloneqq[a_0]_q+\cfrac{q^{a_0}}{[a_1]_{1/q}+\cfrac{q^{-a_1}}{[a_2]_q+\cfrac{q^{a_2}}{[a_3]_{1/q}+\cfrac{q^{-a_3}}{\cdots}}}}
\end{equation*}
in which $\left[a_n\right]_q$ is defined as it was in Section 1 \cite{Morier_Genoud_2019}.
\newline~
\newline
The second definition of the $q$-rationals is based on \textit{algebraic fractional transformations}. Throughout the rest of this section, I will denote $q$-rational numbers defined this way $\alpha(q)/\beta(q)$ rather than $[\alpha/\beta]_q$ so as to avoid confusion when I prove that both definitions are equivalent in Appendix A. Given
\[M=\begin{pmatrix}
\gamma & \alpha \\
\delta & \beta
\end{pmatrix} \in \mathrm{SL}_{2}(\mathbb{Z})\] we can generate a fractional algebraic transformation that maps from \(\mathbb{H} \cup \{\infty\} \cup \mathbb{R}\) to $\mathbb{C}$, in which $\mathbb{H}$ is the upper half-plane on the complex plane, as follows:
\begin{equation*}
M:q\mapsto \begin{dcases}
      \frac{\alpha(q)+q^n \gamma(q)}{\beta(q)+q^n \delta(q)} & \beta+q^n \delta \neq 0 \\
      \infty & \beta+q^n \delta=0
   \end{dcases}
\end{equation*}
in which $q \in \mathbb{H} \cup \infty$ is a formal parameter, $\alpha(q)$, $\beta(q)$, $\gamma(q)$, and $\delta(q)$ are contained in the polynomial ring over the integers $\mathbb{Z}[q]$, and $n = \text{ord }\beta(q) - \text{ord }\delta(q) + 1$ if $\text{ord }\beta(q)>\text{ord }\delta(q)$ and $1$ otherwise \cite{Morier_Genoud_2020}.
\newline~
\newline
Finally, we define
\[
\frac{\alpha(q)}{1(q)} = [\alpha]_q
\]
\begin{equation*}
\frac{(\alpha + \gamma)(q)}{(\beta + \delta)(q)} = \begin{pmatrix}
\gamma & \alpha \\
\delta & \beta
\end{pmatrix}(q)
\end{equation*}
\cite{Morier_Genoud_2020}.
This second definition will helps us observe some critical properties about the $q$-rational numbers. The image of 1 under $M$ is a rational number unique to $M$; this allows us to define a new operation, \textit{Farey addition}, as follows:
\[\frac{\alpha}{\beta} \oplus \frac{\gamma}{\delta} \coloneqq \begin{pmatrix}
\gamma & \alpha \\
\delta & \beta
\end{pmatrix}(1) = \frac{\alpha+\gamma}{\beta+\delta}\]
Starting with $\alpha/\beta=0/1$ and $\gamma/\delta=1/1$, repeated applications of Farey addition form a tree of rational numbers called the \textit{Stern-Brocot tree}:
\begin{center}
\begin{forest}
for tree={alias/.wrap pgfmath arg={a-#1}{id},font=\Large}
[$\frac{1}{2}$
[$\frac{1}{3}$
[$\frac{1}{4}$ 
[$\frac{1}{5}$] 
[$\frac{2}{7}$] ]
[$\frac{2}{5}$
[$\frac{3}{8}$]
[$\frac{3}{7}$] ] ]
[$\frac{2}{3}$
[$\frac{3}{5}$
[$\frac{4}{7}$] 
[$\frac{5}{8}$] ] 
[$\frac{3}{4}$ 
[$\frac{5}{7}$] 
[$\frac{4}{5}$] ] ] ]
\path (current bounding box.north west) node[below right,font=\Large] (tl) {$\frac{0}{1}$}
(tl) foreach \x in {2,3,4,5} {edge (a-\x)}
(current bounding box.north east) node[below right,font=\Large] (tr) {$\frac{1}{1}$}
(tr) foreach \x in {2,10,14,16} {edge (a-\x)}
(a-2) foreach \x in {7,9,11,12} {edge (a-\x)}
(a-3) foreach \x in {6,8} {edge (a-\x)}
(a-10) foreach \x in {13,15} {edge (a-\x)};
\end{forest}
\end{center}
Note that we calculate the tree in layers: after taking $0/1 \oplus 1/1 = 1/2$, we take the Farey sum of the two pairs of consecutive rational numbers thus far on the tree (the pair $0/1,1/2$ and the pair $1/2,1/1$) to obtain $1/3$ and $2/3$. Our tree now contains four pairs of consecutive rational numbers; taking the Farey sums of these pairs, we obtain $1/4$, $2/5$, $3/5$, and $3/4$. We can continue to sum every pair of consecutive rational numbers in our latest iteration of the tree to add layer after layer.
\newline~
\newline
I introduce the following definition to refer to the placement of rational numbers on the tree:
\begin{definition}
The \textit{depth} of a rational number on the Stern-Brocot tree is one less than the number of Farey sums needed to generate it, starting at two consecutive integers; for example, we define the depth of the $0/1$ and $1/1$ to be $-1$, the depth of $1/2$ to be $0$, the depth of $1/3$ and $2/3$ to be $1$, and so on.
\end{definition}
\noindent The tree looks analogous if we start with $\alpha/\beta=m/1$ and $\gamma/\delta=(m+1)/1$ for $m \in \mathbb{Z}$, only each fraction is shifted by $m$.
\newline~
\newline
Next, define the \textit{weighted Farey sum} between $\alpha/\beta$ and $\gamma/\delta$ as the map:
\[\frac{\alpha(q)}{\beta(q)} \oplus_q \frac{\gamma(q)}{\delta(q)} \coloneqq \begin{pmatrix}
\gamma & \alpha \\
\delta & \beta
\end{pmatrix}(q) = \frac{(\alpha+\gamma)(q)}{(\beta+\delta)(q)}\]
\cite{Morier_Genoud_2020}. Starting with $\alpha(q)/\beta(q) = 0/1$ and $\gamma(q)/\delta(q) = 1/1$, we reconstruct the Stern-Brocot tree using repeated applications of the weighted Farey sum to create the $q$\textit{-deformed Stern-Brocot tree}:
\begin{center}
\begin{forest}
for tree={alias/.wrap pgfmath arg={a-#1}{id},font=\footnotesize}
[$\frac{q}{q+1}$
[$\frac{q^2}{q^2+q+1}$
[$\frac{q^3}{q^3+q^2+q+1}$ 
[$\frac{q^4}{q^4+\cdots}$] 
[$\frac{q^4+\cdots}{q^4+\cdots}$] ]
[$\frac{q^3+q^2}{q^3+2q^2+q+1}$
[$\frac{q^4+\cdots}{q^4+\cdots}$]
[$\frac{q^4+\cdots}{q^4+\cdots}$] ] ]
[$\frac{q^2+q}{q^2+q+1}$
[$\frac{q^3+q^2+q}{q^3+q^2+2q+1}$
[$\frac{q^4+\cdots}{q^4+\cdots}$] 
[$\frac{q^4+\cdots}{q^4+\cdots}$] ] 
[$\frac{q^3+q^2+q}{q^3+q^2+q+1}$ 
[$\frac{q^4+\cdots}{q^4+\cdots}$] 
[$\frac{q^4+\cdots}{q^4+\cdots}$] ] ] ]
\path (current bounding box.north west) node[below right,font=\footnotesize] (tl) {$\frac{0}{1}$}
(tl) foreach \x in {2,3,4,5} {edge (a-\x)}
(current bounding box.north east) node[below right,font=\footnotesize] (tr) {$\frac{1}{1}$}
(tr) foreach \x in {2,10,14,16} {edge (a-\x)}
(a-2) foreach \x in {7,9,11,12} {edge (a-\x)}
(a-3) foreach \x in {6,8} {edge (a-\x)}
(a-10) foreach \x in {13,15} {edge (a-\x)};
\end{forest}
\end{center}
We define the depth of a $q$-rational on the $q$-deformed Stern-Brocot tree analogously to before. The tree looks similar if we start with $\alpha(1)/\beta(1) = m/1$ and $\gamma(1)/\delta(1) = (m+1)/1$, except each $q$-rational is shifted by a $q$-integer. Note that at $q = 1$, the $q$-deformed Stern-Brocot tree yields the original Stern-Brocot tree. The choice of $n = \text{ord }\beta(q) - \text{ord }\delta(q) + 1$ if $\text{ord }\beta(q)>\text{ord }\delta(q)$ and $1$ otherwise ensures that our two definitions of the $q$-rational numbers are equivalent, as was proven via a geometric argument by Morier-Genoud and Ovsienko \cite{Morier_Genoud_2020}; I offer an arithmetic proof in Appendix A.
\section{\textbf{Computer Experimentations Used to Find Main Results}}
I wrote a Python script to generate the $q$-deformed Stern-Brocot tree up to a certain depth and calculated the first derivatives of the $q$-rationals that it generated at $q = 1$:
\begin{figure}[H]
\includegraphics[scale=0.37]{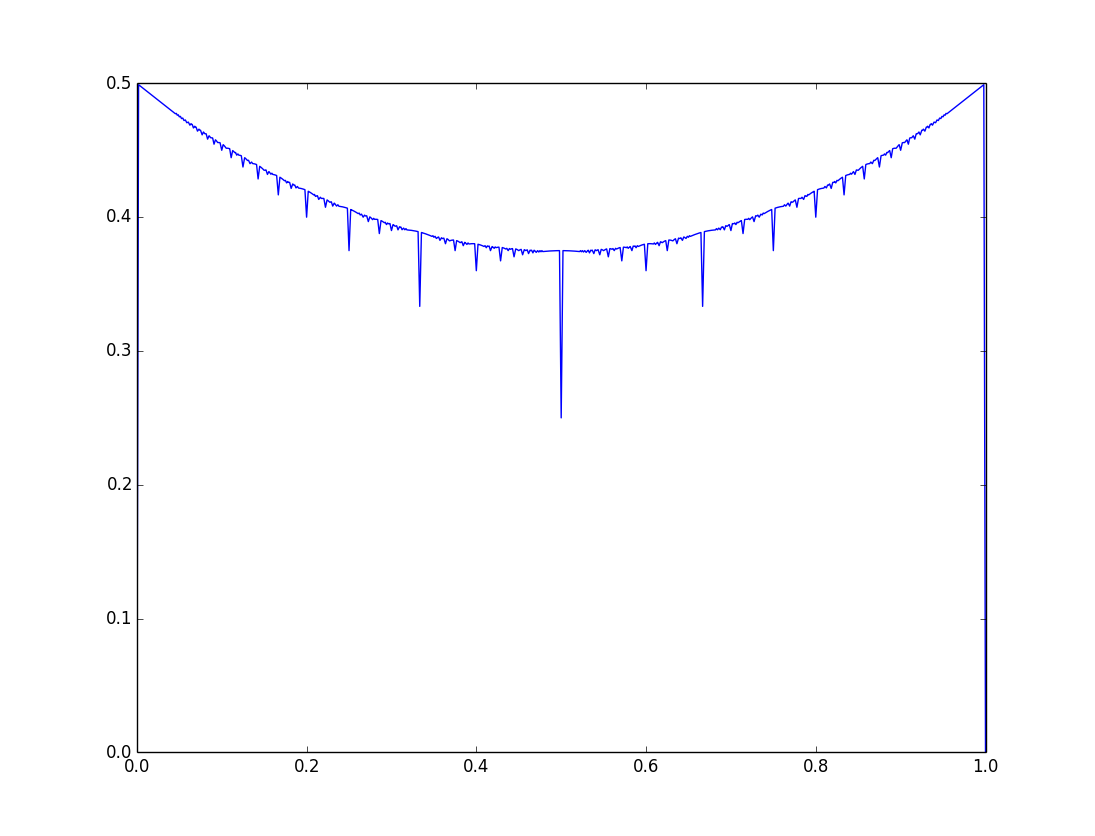}
\centering
\end{figure}
\noindent The curve is roughly quadratic, with spikes that are inversely proportional to the square of the fractions' denominators. Using these observations, I determined the formula for the curve to be of the form
\[\frac{d}{dq}\left[x\right]_q\Bigr|_{q=1} = \alpha x^2 + \beta x + \gamma + \delta f(x)^2\]
in which $f(x)$ is Thomae's function. Using the matrix equality
\[
\begin{pmatrix}
\frac{a_1^2}{b_1^2} & \frac{a_1}{b_1} & 1 & b_1^{-2} \\
\frac{a_2^2}{b_2^2} & \frac{a_2}{b_2} & 1 & b_2^{-2} \\
\frac{a_3^2}{b_3^2} & \frac{a_3}{b_3} & 1 & b_3^{-2} \\
\frac{a_4^2}{b_4^2} & \frac{a_4}{b_4} & 1 & b_4^{-2}
\end{pmatrix}\begin{pmatrix}
\alpha \\
\beta \\
\gamma \\
\delta
\end{pmatrix} =
\begin{pmatrix}
\left.\frac{d}{dq}\frac{a_1(q)}{b_1(q)}\right|_{q=1} \\
\left.\frac{d}{dq}\frac{a_2(q)}{b_2(q)}\right|_{q=1} \\
\left.\frac{d}{dq}\frac{a_3(q)}{b_3(q)}\right|_{q=1} \\
\left.\frac{d}{dq}\frac{a_4(q)}{b_4(q)}\right|_{q=1}
\end{pmatrix}
\]
in which $\frac{a_1(q)}{b_1(q)},\ldots,\frac{a_4(q)}{b_4(q)}$ are arbitrary $q$-rationals, I worked out via linear algebra that $\alpha = 1/2$, $\beta = -1/2$, $\gamma = 1/2$, and $\delta = -1/2$.
\newline
\newline~
I amended my Python script to plot the second derivatives of the $q$-rationals that it generated at $q$=1:
\begin{figure}[H]
\includegraphics[scale=0.5]{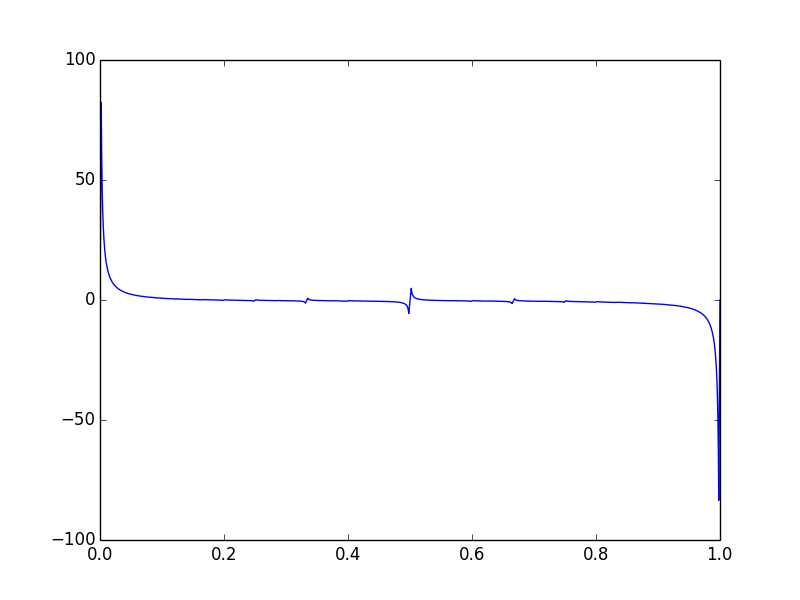}
\centering
\end{figure}
\noindent The resultant curve, which I refer to as \(g(x)\), resembles the function \(x^{-1} - (1-x)^{-1}\). Zooming in closer, we see that this function is actually just the first two terms of a series which approximates our curve:
\begin{figure}[H]
\includegraphics[scale=0.5]{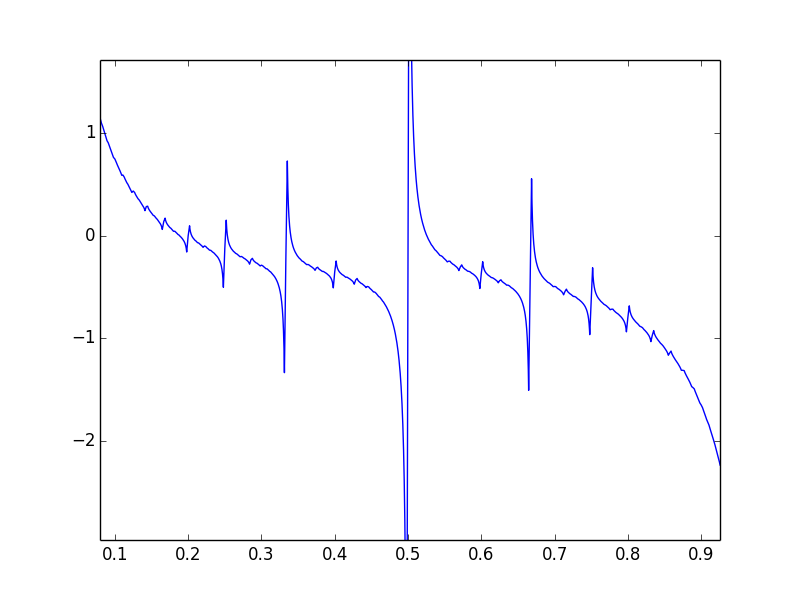}
\centering
\end{figure}
\noindent Completing the series in the graph, we obtain \[g(x) \approx K\sum_{\substack{p,q\in \mathbb{Z} \\ \gcd{p,q}=1}}{\frac{1}{q^4}\frac{1}{x-\frac{p}{q}}}\]
for some positive constant $K$. Unfortunately, this estimation blows up for \(x \in \mathbb{Q}\). The simplest way to avoid this issue is the following:
\[g(x) \approx K\sum_{\substack{p,q\in \mathbb{Z} \\ \gcd{p,q}=1 \\ \frac{p}{q}\neq x}}{\frac{1}{q^4}\frac{1}{x - \frac{p}{q}}}\]
However, a more organic choice is:
\[g\left(\frac{a}{b}\right) \approx K\sum_{\substack{p,q\in \mathbb{Z} \\ \gcd{p,q}=1 \\ b \neq q}}{\frac{1}{q^4}\frac{1}{\frac{a}{b}-\frac{p}{q}}}\]
and $g(x) = \lim_{a/b \rightarrow x}{g(a/b)}$ when $x$ is irrational. Note that since the rational numbers $p/q$ for $\gcd{p,q} = 1$ are symmetric around $a/b$ when $q \mid b$, we have
\[K\sum_{\substack{p,q\in \mathbb{Z} \\ \gcd{p,q}=1 \\ q \mid b \neq q }}{\frac{1}{q^4}\frac{1}{\frac{a}{b}-\frac{p}{q}}} = 0\]
and thus
\[g\left(\frac{a}{b}\right) \approx K\sum_{m=1}^{b-1}{\sum_{n=1}^{\infty}{\sum_{\substack{p \in \mathbb{Z} \\ \gcd{p,nq+m} = 1}}{\frac{1}{(nq+m)^4}\frac{1}{\frac{a}{b}-\frac{p}{nq+m}}}}}\]
As it turns out,
\[
\sum_{n=1}^{\infty}{\sum_{\substack{p \in \mathbb{Z} \\ \gcd{p,nq+m} = 1}}{\frac{1}{(nq+m)^4}\frac{1}{\frac{a}{b}-\frac{p}{nq+m}}}} \approx -\left\langle\frac{m}{a}\right\rangle_ b B_3\left(\frac{m}{b}\right)
\]
and thus \[g\left(\frac{a}{b}\right)\approx -K\sum_{m=1}^{b-1}{\left\langle\frac{m}{a}\right\rangle_ b B_3\left(\frac{m}{b}\right)}\]
in which, as before, \[\left\langle\frac{m}{a}\right\rangle_b \coloneqq \frac{ma^{-1}\:\mathrm{mod}\:b}{b}-\frac{1}{2}\]
where $a^{-1}$ is the inverse of $a$ in the multiplicative group \(\left(\mathbb{Z}\slash b\mathbb{Z}\right)^\times\) and \(x\:\mathrm{mod}\:b\) is the integer between $0$ and $b-1$ corresponding to the congruence class of $x$ modulo $b$. By its definition, the second derivative of \([a/b]]_q\) must have a denominator dividing $b^3$. This is exactly what we get from the approximation $g(a/b)$ above; from here, I knew that \(g(a/b)\) was of the form \[g\left(\frac{a}{b}\right)=W\left(\frac{a}{b}\right)-K\sum_{n=1}^{b-1}{\left\langle\frac{n}{a}\right\rangle_ b B_3\left(\frac{n}{b}\right)}\]
such that \(b^3\times W\left(\frac{a}{b}\right)\in\mathbb{Z}\) and \(K\in\mathbb{N}\). Using the matrix equality
\[
\begin{pmatrix}
\frac{1}{b_1^3} & \frac{a_1}{b_1^3} & \frac{a_1^2}{b_1^3} & \frac{a_1^3}{b_1^3} & \frac{1}{b_1^2} & \frac{a_1}{b_1^2} & \frac{a_1^2}{b_1^2} & \frac{1}{b_1} & \frac{a_1}{b_1} & 1 & \lambda\left(\frac{a_1}{b_1}\right) \\
\frac{1}{b_2^3} & \frac{a_2}{b_2^3} & \frac{a_2^2}{b_2^3} & \frac{a_2^3}{b_2^3} & \frac{1}{b_2^2} & \frac{a_2}{b_2^2} & \frac{a_2^2}{b_2^2} & \frac{1}{b_2} & \frac{a_2}{b_2} & 1 & \lambda\left(\frac{a_2}{b_2}\right) \\
\vdots & \vdots & \vdots & \vdots & \vdots & \vdots & \vdots & \vdots & \vdots & \vdots & \vdots \\
\frac{1}{b_{11}^3} & \frac{a_{11}}{b_{11}^3} & \frac{a_{11}^2}{b_{11}^3} & \frac{a_{11}^3}{b_{11}^3} & \frac{1}{b_{11}^2} & \frac{a_{11}}{b_{11}^2} & \frac{a_{11}^2}{b_{11}^2} & \frac{1}{b_{11}} & \frac{a_{11}}{b_{11}} & 1 & \lambda\left(\frac{a_{11}}{b_{11}}\right)
\end{pmatrix}\begin{pmatrix}
c_1 \\
c_2 \\
\vdots \\
c_{11}
\end{pmatrix} =
\]
\[
\begin{pmatrix}
\left.\frac{d^2}{dq^2}\frac{a_1(q)}{b_1(q)}\right|_{q=1} \\
\left.\frac{d^2}{dq^2}\frac{a_2(q)}{b_2(q)}\right|_{q=1} \\
\vdots \\
\left.\frac{d^2}{dq^2}\frac{a_{11}(q)}{b_{11}(q)}\right|_{q=1}
\end{pmatrix}
\]
in which \[\lambda\left(\frac{a}{b}\right) = \sum_{n=1}^{b-1}{\left\langle\frac{n}{a}\right\rangle_ b B_3\left(\frac{n}{b}\right)}\] and $\frac{a_1(q)}{b_1(q)},\ldots,\frac{a_{11}(q)}{b_{11}(q)}$ are arbitrary $q$-rationals, I worked out that \(K=20\) and \(W(x)=F(x) + f(x)^2G(x)\), in which \(F(x) = x^3/3-x^2+5x/3-1\), \(f(x)\) is Thomae's function, and \(G(x) = 1 - x\).
\section{\normalfont \textbf{Proof Schema} \(\mid\) Proving Identities for Arbitrary-Order Derivatives of the $q$-Rationals}
In this section, I will provide a structure by which we can prove an identity for the $n$th derivative of the $q$-rationals at $q = 1$. The structure requires that we have already proven the following identities.
\newline~
\newline
\textit{The closed-form expressions }$f_1(x),\ldots,f_{n-1}(x),g_1(x)\ldots,g_{n-1}(x)$\textit{ satisfy that, for all }$q$\textit{-rationals }$\left[\frac{a}{b}\right]_q=\frac{a(q)}{b(q)}$\textit{, we have}
\[
\left.\frac{d}{dq}a(q)\right|_{q=1}=f_1\left(\frac{a}{b}\right),\ldots,\left.\frac{d^{n-1}}{d^{n-1}q}a(q)\right|_{q=1}=f_{n-1}\left(\frac{a}{b}\right)
\]
\textit{and}
\[
\left.\frac{d}{dq}b(q)\right|_{q=1}=g_1\left(\frac{a}{b}\right),\ldots,\left.\frac{d^{n-1}}{d^{n-1}q}b(q)\right|_{q=1}=g_{n-1}\left(\frac{a}{b}\right)
\]
I now introduce the notion of a $q$-rational's \textit{lineage of order }$m$:
\begin{definition}
Given $\left[\frac{\alpha}{\beta}\right]_q=\frac{\alpha(q)}{\beta(q)}$, define its \textit{lineage of order }$m$
\[
\frac{a_1(q)}{b_1(q)},\ldots,\frac{a_m(q)}{b_m(q)}
\]
such that $\frac{a_m(q)}{b_m(q)} = \frac{\alpha(q)}{\beta(q)}$ and the parents of $\frac{a_n(q)}{b_n(q)}$ on the Stern-Brocot tree are $\frac{a_{n-1}(q)}{b_{n-1}(q)}$ and $\frac{a_{\zeta_n}(q)}{b_{\zeta_n}(q)}$ for some positive integer $\zeta_n < n - 1$. We call a lineage of order $m$ \textit{vanishing} if $\frac{a_1(q)}{b_1(q)} \in \mathbb{Z} \cup \infty$ and \textit{non-vanishing} otherwise.
\end{definition}
\noindent Examples of lineages of orders 4 and 5 are shown below.
\begin{figure}[H]
\begin{floatrow}
\begin{subfigure}[t]{0.5\linewidth}
\begin{center}
\begin{forest}
for tree={alias/.wrap pgfmath arg={a-#1}{id},font=\Large}
[$\frac{a_2}{b_2}$
[$\frac{a_3}{b_3}$
[] 
[$\frac{a_4}{b_4}$] ]
[
[]
[] ] ]
\path (current bounding box.north west) node[below right,font=\Large] (tl) {$\frac{a_1}{b_1}$}
(tl) foreach \x in {2,3,4} {edge (a-\x)}
(current bounding box.north east) node[below right,font=\Large] (tr) {}
(tr) foreach \x in {2,6,8} {edge (a-\x)}
(a-2) foreach \x in {5,7} {edge (a-\x)};
\draw (tl) edge [very thick] (a-2)
(tl) edge [very thick] (a-3)
(a-2) edge [very thick] (a-3)
(a-2) edge [very thick] (a-5)
(a-3) edge [very thick] (a-5);
\end{forest}
\end{center}
\caption{Fig. 4.1: A lineage of order 4.}
\end{subfigure}
\begin{subfigure}[t]{0.5\linewidth}
\begin{center}
\begin{forest}
for tree={alias/.wrap pgfmath arg={a-#1}{id},font=\Large}
[$\frac{a_2}{b_2}$
[
[
[] 
[] ]
[
[]
[] ] ]
[$\frac{a_3}{b_3}$
[
[] 
[] ] 
[$\frac{a_4}{b_4}$
[] 
[$\frac{a_5}{b_5}$] ] ] ]
\path (current bounding box.north west) node[below right,font=\Large] (tl) {}
(tl) foreach \x in {2,3,4,5} {edge (a-\x)}
(current bounding box.north east) node[below right,font=\Large] (tr) {$\frac{a_1}{b_1}$}
(tr) foreach \x in {2,10,14,16} {edge (a-\x)}
(a-2) foreach \x in {7,9,11,12} {edge (a-\x)}
(a-3) foreach \x in {6,8} {edge (a-\x)}
(a-10) foreach \x in {13,15} {edge (a-\x)};
\draw (tr) edge [very thick] (a-2)
(tr) edge [very thick] (a-10)
(tr) edge [very thick] (a-14)
(tr) edge [very thick] (a-16)
(a-2) edge [very thick] (a-10)
(a-10) edge [very thick] (a-14)
(a-14) edge [very thick] (a-16);
\end{forest}
\end{center}
\caption{Fig. 4.2: A lineage of order 5.}
\end{subfigure}
\end{floatrow}
\end{figure}
\noindent There exist sequences $\mathfrak{F}_1(q),\ldots,\mathfrak{F}_m(q),\mathfrak{G}_1(q),\ldots,\mathfrak{G}_m(q)$ for which we can re-express the $q$-rationals along the lineage of order $m$ of $[\alpha/\beta]_q$ as
\begin{equation}
\frac{a_n(q)}{b_n(q)} = \frac{\mathfrak{F}_n(q) a_1(q)+\mathfrak{G}_n(q) a_2(q)}{\mathfrak{F}_n(q) b_1(q)+\mathfrak{G}_n(q) b_2(q)}
\tag{4.1}
\end{equation}
These sequences can be calculated as follows: let $\mathfrak{F}_1(q) = 1,\mathfrak{F}_2(q) = 0,\mathfrak{G}_1(q) = 0,\mathfrak{G}_2(q) = 1$, and define $\mathfrak{F}_n$ and $\mathfrak{G}_n$ by the recurrence relation
\begin{alignat*}{3}
& \mathfrak{F}_n(q) = \mathfrak{F}_{n-1}(q) + q^{\xi_n}\mathfrak{F}_{\zeta_n}(q) && \quad \mathfrak{G}_n(q) = \mathfrak{G}_{n-1}(q) + q^{\xi_n}\mathfrak{G}_{\zeta_n}(q)
\end{alignat*}
in which $\zeta_n \in \mathbb{Z}_{>0}$ is chosen such that $\frac{a_{\zeta_n}(q)}{a_{\zeta_n}(q)}$ is the parent of $\frac{a_n(q)}{a_n(q)}$ on the $q$-deformed Stern-Brocot tree that is not $\frac{a_{n-1}(q)}{a_{n-1}(q)}$, and 
\begin{equation*}\
{\small
  \xi_n=
  \begin{cases}
    \begin{cases}
      \text{ord }b_{\zeta_n}(q) - \text{ord }b_{n-1}(q) + 1 & \text{if }b_{\zeta_n}(q) > \text{ord }b_{n-1}(q)\\
      1 & \text{otherwise}
    \end{cases}
    &\text{if }\frac{a_{n-1}(q)}{b_{n-1}(q)} > \frac{a_{\zeta_n}(q)}{b_{\zeta_n}(q)}\\
    \begin{cases}
     \text{ord }b_{n-1}(q) - \text{ord }b_{\zeta_n}(q) + 1 & \text{if }\text{ord }b_{n-1}(q) > b_{\zeta_n}(q)\\
      1 & \text{otherwise}
    \end{cases}
    &\text{otherwise}
  \end{cases}}
\end{equation*}
To show why our choices of $\mathfrak{F}_1(q),\ldots,\mathfrak{F}_m(q),\mathfrak{G}_1(q),\ldots,\mathfrak{G}_m(q)$ satisfy Equation 4.1, first let
\[L_n(q) = \frac{\mathfrak{F}_{\zeta_n}(q) a_1(q)+\mathfrak{G}_{\zeta_n}(q) a_2(q)}{\mathfrak{F}_{\zeta_n}(q) b_1(q)+\mathfrak{G}_{\zeta_n}(q) b_2(q)}\]
Our choice of $\xi_n$ above ensures that
\begin{equation*}
L_n(q) = \begin{cases}
L_{\zeta_n}(q) \oplus_q L_{n-1}(q)
& \text{if }L_{n-1}(1)>L_{\zeta_n}(1)\\
L_{n-1}(q) \oplus_q L_{\zeta_n}(q){b_{\zeta_n}}
    & \text{otherwise}
\end{cases}
\end{equation*}
Since $L_1(q)=a_1(q)/b_1(q)$ and  $L_2(q)=a_2(q)/b_2(q)$ by definition, we can show via induction that $L_n(q)=a_n(q)/b_n(q)$. We find that the sequences we chose satisfy Theorem 4.2, and we can use them to rewrite the tree above in terms of $a_1/b_1$ and $a_2/b_2$:
\begin{figure}[H]
\begin{floatrow}
\begin{subfigure}[t]{0.5\linewidth}
\begin{center}
\begin{forest}
for tree={alias/.wrap pgfmath arg={a-#1}{id},font=\Large}
[$\frac{a_2}{b_2}$
[$\frac{a_1+qa_2}{b_1+qb_2}$
[] 
[$\frac{a_1+(q+q^2)a_2}{b_1+(q+q^2)b_2}$] ]
[
[]
[] ] ]
\path (current bounding box.north west) node[below right,font=\Large] (tl) {$\frac{a_1}{b_1}$}
(tl) foreach \x in {2,3,4} {edge (a-\x)}
(current bounding box.north east) node[below right,font=\Large] (tr) {}
(tr) foreach \x in {2,6,8} {edge (a-\x)}
(a-2) foreach \x in {5,7} {edge (a-\x)};
\draw (tl) edge [very thick] (a-2)
(tl) edge [very thick] (a-3)
(a-2) edge [very thick] (a-3)
(a-2) edge [very thick] (a-5)
(a-3) edge [very thick] (a-5);
\end{forest}
\end{center}
\caption{Fig. 4.3: A relabeling of Fig. 4.1.}
\end{subfigure}
\begin{subfigure}[t]{0.5\linewidth}
\captionsetup{singlelinecheck=off}
\begin{center}
\begin{forest}
for tree={alias/.wrap pgfmath arg={a-#1}{id},font=\Large}
[$\frac{a_2}{b_2}$
[
[
[] 
[] ]
[
[]
[] ] ]
[$\frac{a_1+\mathfrak{G}_3(q)a_2}{b_1+\mathfrak{G}_3(q)b_2}$
[
[] 
[] ] 
[$\frac{a_1+\mathfrak{G}_4(q)a_2}{b_1+\mathfrak{G}_4(q)b_2}$
[] 
[$\frac{a_1+\mathfrak{G}_5(q)a_2}{b_1+\mathfrak{G}_5(q)b_2}$] ] ] ]
\path (current bounding box.north west) node[below right,font=\Large] (tl) {}
(tl) foreach \x in {2,3,4,5} {edge (a-\x)}
(current bounding box.north east) node[below right,font=\Large] (tr) {$\frac{a_1}{b_1}$}
(tr) foreach \x in {2,10,14,16} {edge (a-\x)}
(a-2) foreach \x in {7,9,12} {edge (a-\x)}
(a-3) foreach \x in {6,8} {edge (a-\x)}
(a-10) foreach \x in {13} {edge (a-\x)};
\draw (tr) edge [very thick] (a-2)
(tr) edge [very thick] (a-10)
(tr) edge [very thick] (a-14)
(tr) edge [very thick] (a-16)
(a-2) edge [very thick] (a-10)
(a-2) edge [thick, dotted] (a-11)
(a-10) edge [very thick] (a-14)
(a-10) edge [thick, dotted] (a-15)
(a-14) edge [very thick] (a-16);
\end{forest}
\end{center}
\caption{Fig. 4.4: A relabeling of Fig. 4.2, in which \\
\begin{itemize}
\item $\mathfrak{G}_3(q)=q^{n-1}$
\item $\mathfrak{G}_4(q)=q^{n}+q^{n-1}$
\item $\mathfrak{G}_5(q)=q^{n+1}+q^{n}+q^{n-1}$
\end{itemize}}
\end{subfigure}
\end{floatrow}
\end{figure}
\noindent Let $f_n = \mathfrak{F}_n(1)$ and $g_n = \mathfrak{G}_n(1)$ for all $n \in [1,m]$. Going forward, I will refer to $f_1,\ldots,f_m,g_1,\ldots,g_m$ as the \textit{weights} of a given lineage of order $m$. Note that, in our proofs in Sections 5 and 6, I express the $q$-rationals along the lineage of order $m$ of $[\alpha/\beta]_q$ in terms of $a(q) = a_{m-1}(q)$, $b(q) = b_{m-1}(q)$, $c(q) = a_1(q)$, and $d(q) = b_1(q)$. Relabelings of the trees in Fig. 4.1 and Fig. 4.2 using this convention are given below.
\begin{figure}[H]
\begin{floatrow}
\begin{subfigure}[t]{0.5\linewidth}
\begin{center}
\begin{forest}
for tree={alias/.wrap pgfmath arg={a-#1}{id},font=\Large}
[$\frac{a-c}{b-d}$
[$\frac{a}{b}$
[] 
[$\frac{(q+1)a-qc}{(q+1)b-qd}$] ]
[
[]
[] ] ]
\path (current bounding box.north west) node[below right,font=\Large] (tl) {$\frac{c}{d}$}
(tl) foreach \x in {2,3,4} {edge (a-\x)}
(current bounding box.north east) node[below right,font=\Large] (tr) {}
(tr) foreach \x in {2,6,8} {edge (a-\x)}
(a-2) foreach \x in {5,7} {edge (a-\x)};
\draw (tl) edge [very thick] (a-2)
(tl) edge [very thick] (a-3)
(a-2) edge [very thick] (a-3)
(a-2) edge [very thick] (a-5)
(a-3) edge [very thick] (a-5);
\end{forest}
\end{center}
\caption{Fig. 4.5: An alternate re-labeling of Fig. 4.1.}
\end{subfigure}
\begin{subfigure}[t]{0.5\linewidth}
\begin{center}
\begin{forest}
for tree={alias/.wrap pgfmath arg={a-#1}{id},font=\Large}
[$\frac{a-(q^{n}+q^{n-1})c}{b-(q^{n}+q^{n-1})d}$
[
[
[] 
[] ]
[
[]
[] ] ]
[$\frac{a-q^{n}c}{b-q^{n}d}$
[
[] 
[] ] 
[$\frac{a}{b}$
[] 
[$\frac{a+q^{n+1}c}{b+q^{n+1}d}$] ] ] ]
\path (current bounding box.north west) node[below right,font=\Large] (tl) {}
(tl) foreach \x in {2,3,4,5} {edge (a-\x)}
(current bounding box.north east) node[below right,font=\Large] (tr) {$\frac{c}{d}$}
(tr) foreach \x in {2,10,14,16} {edge (a-\x)}
(a-2) foreach \x in {7,9,11,12} {edge (a-\x)}
(a-3) foreach \x in {6,8} {edge (a-\x)}
(a-10) foreach \x in {13,15} {edge (a-\x)};
\draw (tr) edge [very thick] (a-2)
(tr) edge [very thick] (a-10)
(tr) edge [very thick] (a-14)
(tr) edge [very thick] (a-16)
(a-2) edge [very thick] (a-10)
(a-10) edge [very thick] (a-14)
(a-14) edge [very thick] (a-16);
\end{forest}
\end{center}
\caption{Fig. 4.6: An alternate re-labeling of Fig. 4.2.}
\end{subfigure}
\end{floatrow}
\end{figure}
\noindent Note that a factor of $q$ is canceled out from their respective numerators and denominators to yield $a/b$ and $c/d$ in Fig. 4.5; I will similarly reduce $q$-rationals on the $q$-deformed Stern-Brocot tree throughout the rest of the paper. Also note that, while I instead express the $q$-rationals in terms of $a_1(q)$, $a_2(q)$, $b_1(q)$, and $b_2(q)$ in the proof structure, the substitution
\begin{alignat*}{3}
& a(q) = \mathfrak{F}_{m-1}(q) a_1(q) + \mathfrak{G}_{m-1}(q) a_2(q) && \quad b(q) = \mathfrak{F}_{m-1}(q) b_1(q) + \mathfrak{G}_{m-1}(q) b_2(q) \\
& c(q) = \mathfrak{F}_{1}(q) a_1(q) + \mathfrak{G}_{1}(q) a_2(q) && \quad d(q) = \mathfrak{F}_{1}(q) b_1(q) + \mathfrak{G}_{1}(q) b_2(q)
\end{alignat*}
shows that the convention above does not come with loss of generality. We now prove a theorem that will be instrumental to our proof structure:
\begin{theorem}
Let
\[\Delta_i\left(\frac{\alpha(q)}{\beta(q)}\right)=\left.\left(\frac{\alpha^{(i)}(q)}{\beta(q)}+(-1)^{i}\frac{\alpha(q)\beta^{(i)}(q)}{\beta^{i+1}(q)}\right)\right|_{q=1}\]
for $\alpha(q),\beta(q) \in \mathbb{Z}[q]$. Denote the lineage of order $m$ of $\left[\frac{\alpha}{\beta}\right]_q=\frac{\alpha(q)}{\beta(q)}$ as $\frac{a_1(q)}{b_1(q)},\ldots,\break\frac{a_m(q)}{b_m(q)}$, and suppose this lineage is non-vanishing. Furthermore, denote the lineage's weights $f_1,\ldots,f_m,g_1,\ldots,g_m$. $\Delta_{m-3}\left(\frac{\alpha(q)}{\beta(q)}\right)$ is linearly dependent on 
$\left(\frac{b_i(1)}{b(1)}\right)^{m-2}\Delta_{m-3}\left(\frac{a_1(q)}{b_1(q)}\right),\ldots,\left(\frac{b_{m-1}(1)}{b(1)}\right)^{m-2}\Delta_{m-3}\left(\frac{a_{m-1}(q)}{b_{m-1}(q)}\right)$; moreover, given
\begin{equation*}
\Delta_{m-3}\left(\frac{\alpha(q)}{\beta(q)}\right) = \sum_{i=1}^{m-1}{C_i\left(\frac{b_i(1)}{b(1)}\right)^{m-2}\Delta_{m-3}\left(\frac{a_i(q)}{b_i(q)}\right)}
\end{equation*}
we have that
\begin{equation}
C_i = \prod_{\substack{1\leq n < m \\ n \neq i}}{\frac{f_m g_n-f_n g_m}{f_i g_n-f_n g_i}}
\tag{4.2}
\end{equation}
\end{theorem}
\begin{proof}
Define $\frac{a_{i}(q)}{b_{i}(q)}$, $f_i$, and $g_i$ for all $i\in [1,m]$ as outlined in the proof statement. Taking Equation 4.1 at $q = 1$, we have that
\[\frac{a_{i}(1)}{b_{i}(1)} = \frac{f_i a_1(1) + g_i a_2 (1)}{f_i b_1(1) + g_i b_2(1)}\]
Define $a_1 = a_1(1), a_2 = a_2(1), b_1 = b_1(1), b_2 = b_2(1)$. We rewrite the coefficient of $a_1^{(i)}(q)$ in Equation 4.2 as
\[
f_m(f_m b_1 + g_m b_2)^{m-3} = \sum_{i=1}^{m-1}{C_i f_i(f_i b_1 + g_i b_2)^{m-3}}
\]
and the coefficient of $a_2^{(i)}(q)$ in Equation 4.2 as
\[
f_m(f_m b_1 + g_m b_2)^{m-3} = \sum_{i=1}^{m-1}{C_i g_i (f_i b_1 + g_i b_2)^{m-3}}
\]
These equations take the form
\[
(\phi f_m + \psi g_m)(f_m b_1 + g_m b_2)^{m-3} = \sum_{i=1}^{m-1}{(\phi f_m + \psi g_m)C_i(f_i b_1 + g_i b_2)^{m-3}}
\]
in which either $(\phi,\psi) = (1,0)$ or $(\phi,\psi) = (0,1)$. We can re-express the equation above as
\begin{equation*}
\begin{gathered}
\sum_{i=1}^{m-1}{\sum_{j=0}^{m-2}{C_i\left({m-3 \choose j }\phi b_2+{m-3 \choose j - 1}\psi b_1\right)b_1^{j-1}b_2^{m-3-j}f_i^{j}g_i^{m-2-j} }} = \\
\sum_{j=0}^{m-2}{\left({m-3 \choose j }\phi b_2+{m-3 \choose j - 1}\psi b_1\right)b_1^{j-1}b_2^{m-3-j}f_m^{j}g_m^{m-2-j}}
\end{gathered}
\end{equation*}
We look to find a unique solution for $C_1,\ldots,C_{m-1}$ that satisfies the equation above for either choice of $(\phi,\psi)$. Since our lineage is non-vanishing, $b_1 \neq 0$; in no lineage does $b_2 = 0$, and thus the term
\[\left({m-3 \choose j }\phi b_2+{m-3 \choose j - 1}\psi b_1\right)b_1^{j-1}b_2^{m-3-j}\]
cannot be zero for either choice of $(\phi,\psi)$. We conclude that we do not shrink our solution space by cancelling it on both sides. In matrix form, we are left with
\begin{equation*}
\begin{gathered}
\underbrace{\begin{pmatrix}
1 & 1 & 1 & 1 & \cdots & 1 \\
\frac{f_1}{g_1} & \frac{f_2}{g_2} & \frac{f_3}{g_3} & \frac{f_4}{g_4} & \cdots & \frac{f_{m-1}}{g_{m-1}} \\
\frac{f_1^2}{g_1^2} & \frac{f_2^2}{g_2^2} & \frac{f_3^2}{g_3^2} & \frac{f_4^2}{g_4^2} & \cdots & \frac{f_{m-1}^2}{g_{m-1}^2} \\
\vdots & \vdots & \vdots & \ddots & \vdots \\
\frac{f_1^{m-2}}{g_1^{m-2}} & \frac{f_2^{m-2}}{g_2^{m-2}} & \frac{f_2^{m-2}}{g_2^{m-2}} & \frac{f_3^{m-2}}{g_3^{m-2}} & \cdots & \frac{f_{m-1}^{m-2}}{g_{m-1}^{m-2}}
\end{pmatrix}}_{V}\times \\
\begin{pmatrix}
\frac{g_1^{m-2}}{g_m^{m-2}} & 0 & 0 & \cdots & 0 \\
0 & \frac{g_2^{m-2}}{g_m^{m-2}} & 0 & \cdots & 0 \\
0 & 0 & \frac{g_3^{m-2}}{g_m^{m-2}} & \cdots & 0 \\
\vdots & \vdots & \vdots & \ddots & \vdots \\
0 & 0 & 0 & \cdots & \frac{g_{m-1}^{m-2}}{g_m^{m-2}}
\end{pmatrix}\begin{pmatrix}
C_1 \\ C_2 \\ C_3 \\ \vdots \\ C_{m-1}
\end{pmatrix} =
\underbrace{\begin{pmatrix}
1 \\ \frac{f_m}{g_m} \\ \frac{f_m^2}{g_m^2} \\ \vdots \\ \frac{f_m^{m-2}}{g_m^{m-2}}
\end{pmatrix}}_{K}
\end{gathered}
\end{equation*}
Note that $V$ is the transpose of a Vandermonde matrix; the $k$th row of its inverse must be the coefficients of the Lagrange polynomial
\[
L_{k}(x) = L_{0,k} + L_{1,k}x + L_{2,k}x^2 + \cdots + L_{m-2,k}x^{j} = \prod_{\substack{0 \leq n < m \\ n \neq k-1}}{\frac{x-f_n/g_n}{f_{k-1}/g_{k-1}-f_n/g_n}}
\]
We see that
\begin{equation*}
V^{-1}K = \begin{pmatrix}
L_1\left(f_m/g_m\right) \\
L_{2}\left(f_m/g_m\right) \\
L_{3}\left(f_m/g_m\right) \\
\vdots \\
L_{m-1}\left(f_m/g_m\right)
\end{pmatrix}
\end{equation*}
and thus
\begin{equation*}
\begin{gathered}
C_i = \left(\frac{g_m}{g_i}\right)^{m-2}\prod_{\substack{1\leq n < m \\ n \neq i}}{\frac{f_m/g_m-f_n/g_n}{f_i/g_i-f_n/g_n}} = \prod_{\substack{1\leq n < m \\ n \neq i}}{\frac{f_m g_n-f_n g_m}{f_i g_n-f_n g_i}}
\end{gathered}
\end{equation*}
Finally, we must show that this result ensures that $b_1^{(i)}(q)$ and $b_2^{(i)}(q)$ are equal on both sides of Equation 4.2. We rewrite the coefficient of $b_1^{(i)}(q)$ in Equation 4.2 as
\[
f_m(f_m a_1+g_m a_2) = \sum_{i=1}^{m-1}{C_i f_i(f_i a_1 + g_i a_2)}
\]
and the coefficient of $b_2^{(i)}(q)$ in Equation 4.2 as
\[
g_m(f_m a_1+g_m a_2) = \sum_{i=1}^{m-1}{C_i g_i(f_i a_1 + g_i a_2)}
\]
These equations take the form
\[
(\phi f_m + \psi g_m)(f_m a_1 + g_m a_2) = \sum_{i=1}^{m-1}{(\phi f_m + \psi g_m)C_i(f_i a_1 + g_i a_2)}
\]
in which either $\phi = 1, \psi = 0$ or $\phi = 0, \psi = 1$. We can re-express the equation above as
\begin{equation*}
\begin{gathered}
\sum_{i=1}^{m-1}{\left(\phi a_1 f_i^2 + (\phi a_2 + \psi a_1)f_i g_i + \psi a_2 g_i^2\right)} = \phi a_1 f_m^2 + (\phi a_2 + \psi a_1)f_m g_m + \psi a_2 g_m^2
\end{gathered}
\end{equation*}
In matrix form, we have
\begin{equation*}
\begin{gathered}
\begin{pmatrix}
1 & 1 & 1 & 1 & \cdots & 1 \\
\frac{f_1}{g_1} & \frac{f_2}{g_2} & \frac{f_3}{g_3} & \frac{f_4}{g_4} & \cdots & \frac{f_{m-1}}{g_{m-1}} \\
\frac{f_1^2}{g_1^2} & \frac{f_2^2}{g_2^2} & \frac{f_3^2}{g_3^2} & \frac{f_4^2}{g_4^2} & \cdots & \frac{f_{m-1}^2}{g_{m-1}^2}
\end{pmatrix} \times \\
\begin{pmatrix}
\frac{g_1^{m-2}}{g_m^{m-2}} & 0 & 0 & \cdots & 0 \\
0 & \frac{g_2^{m-2}}{g_m^{m-2}} & 0 & \cdots & 0 \\
0 & 0 & \frac{g_3^{m-2}}{g_m^{m-2}} & \cdots & 0
\end{pmatrix}
\begin{pmatrix}
C_1 \\ C_2 \\ C_3
\end{pmatrix} =
\begin{pmatrix}
1 \\ \frac{f_m}{g_m} \\ \frac{f_m^2}{g_m^2}
\end{pmatrix}
\end{gathered}
\end{equation*}
in which we've divided the first row on both sides by $\phi a_1$, the second row on both sides by $\phi a_2 + \psi a_1$, and the third row on both sides by $\psi a_2$. (Since our lineage is non-vanishing, $a_1 \neq 0$; in no lineage does $a_2 = 0$, and thus the terms above cannot be zero for either choice of $(\phi,\psi)$. We conclude that we do not shrink our solution space by cancelling them out on both sides.) We discover that we have already proven that our choices of $C_i$ satisfy this matrix equality, as it is just a restatement of our initial matrix equality over its first three rows. The proof is now complete.
\end{proof}
\noindent I now introduce the concept of a \textit{slab of order }$n$\textit{ starting at }$m$ on the $q$-deformed Stern-Brocot tree.
\begin{definition}
A \textit{slab of order }$n$\textit{ starting at }$m$\textit{ on the }$q$\textit{-deformed Stern-Brocot tree} is the set of $q$-rationals of depth $m$ through $m+n-1$ on the tree. We refer to those $q$-rationals at depth $m+n-1$ as being on the last layer of the slab.
\end{definition}
\noindent Let $j$ be the order of the derivative that we wish to prove satisfies a given identity. Recall that to use the schema, we must already have the follow information:
\newline~
\newline
\textit{Given }$\left[\frac{a}{b}\right]_q=\frac{a(q)}{b(q)}$\textit{, we have}
\[
\left.\frac{d}{dq}a(q)\right|_{q=1}=f_1\left(\frac{a}{b}\right),\ldots,\left.\frac{d^{j-1}}{d^{j-1}q}a(q)\right|_{q=1}=f_{j-1}\left(\frac{a}{b}\right)
\]
\textit{and}
\[
\left.\frac{d}{dq}b(q)\right|_{q=1}=g_1\left(\frac{a}{b}\right),\ldots,\left.\frac{d^{j-1}}{d^{j-1}q}b(q)\right|_{q=1}=g_{j-1}\left(\frac{a}{b}\right)
\]
To prove our identity, we will use the slab of order $j+3$ starting at depth $-1$ on the $q$-deformed Stern-Brocot tree as a base case, and assume the identity holds for the slab of order $j+3$ starting at $m$. We break up the orientation of the $q$-integer on the last layer of the slab relative to its lineage of order $j+3$ into $2^{j+1}$ cases. For example, the four possibilities for $j=1$ are shown below:
\begin{figure}[H]
\begin{floatrow}
\begin{subfigure}[t]{0.5\linewidth}
\begin{center}
\begin{forest}
for tree={alias/.wrap pgfmath arg={a-#1}{id},font=\Large}
[$\frac{a-q^{n}c}{b-q^{n}d}$
[{\large \dots}
[{\large \dots}] 
[{\large \dots}] ]
[$\frac{a}{b}$
[{\large \dots}]
[$\frac{a+q^{n+1}c}{b+q^{n+1}d}$] ] ]
\path (current bounding box.north west) node[below right,font=\Large] (tl) {{\large \dots}}
(tl) foreach \x in {2,3,4} {edge (a-\x)}
(current bounding box.north east) node[below right,font=\Large] (tr) {$\frac{c}{d}$}
(tr) foreach \x in {2,6,8} {edge (a-\x)}
(a-2) foreach \x in {5,7} {edge (a-\x)};
\draw (tr) edge [very thick] (a-2)
(tr) edge [very thick] (a-6)
(tr) edge [very thick] (a-8)
(a-2) edge [very thick] (a-6)
(a-6) edge [very thick] (a-8);
\end{forest}
\end{center}
\end{subfigure}
\begin{subfigure}[t]{0.5\linewidth}
\begin{center}
\begin{forest}
for tree={alias/.wrap pgfmath arg={a-#1}{id},font=\Large}
[$\frac{a-q^{n}c}{b-q^{n}d}$
[{\large \dots}
[{\large \dots}] 
[{\large \dots}] ]
[$\frac{a}{b}$
[$\frac{(q+1)a-q^{n}c}{(q+1)b-q^{n}d}$]
[{\large \dots}] ] ]
\path (current bounding box.north west) node[below right,font=\Large] (tl) {{\large \dots}}
(tl) foreach \x in {2,3,4} {edge (a-\x)}
(current bounding box.north east) node[below right,font=\Large] (tr) {$\frac{c}{d}$}
(tr) foreach \x in {2,6,8} {edge (a-\x)}
(a-2) foreach \x in {5,7} {edge (a-\x)};
\draw (tr) edge [very thick] (a-2)
(tr) edge [very thick] (a-6)
(a-2) edge [very thick] (a-6)
(a-2) edge [very thick] (a-7)
(a-6) edge [very thick] (a-7);
\end{forest}
\end{center}
\end{subfigure}
\end{floatrow}
\end{figure}
\begin{figure}[H]
\begin{floatrow}
\begin{subfigure}[t]{0.5\linewidth}
\begin{center}
\begin{forest}
for tree={alias/.wrap pgfmath arg={a-#1}{id},font=\Large}
[$\frac{a-c}{b-d}$
[$\frac{a}{b}$
[{\large \dots}] 
[$\frac{(q+1)a-qc}{(q+1)b-qd}$] ]
[{\large \dots}
[{\large \dots}]
[{\large \dots}] ] ]
\path (current bounding box.north west) node[below right,font=\Large] (tl) {$\frac{c}{d}$}
(tl) foreach \x in {2,3,4} {edge (a-\x)}
(current bounding box.north east) node[below right,font=\Large] (tr) {{\large \dots}}
(tr) foreach \x in {2,6,8} {edge (a-\x)}
(a-2) foreach \x in {5,7} {edge (a-\x)};
\draw (tl) edge [very thick] (a-2)
(tl) edge [very thick] (a-3)
(a-2) edge [very thick] (a-3)
(a-2) edge [very thick] (a-5)
(a-3) edge [very thick] (a-5);
\end{forest}
\end{center}
\end{subfigure}
\begin{subfigure}[t]{0.5\linewidth}
\begin{center}
\begin{forest}
for tree={alias/.wrap pgfmath arg={a-#1}{id},font=\Large}
[$\frac{a-c}{b-d}$
[$\frac{a}{b}$
[$\frac{qa+c}{qb+d}$] 
[{\large \dots}] ]
[{\large \dots}
[{\large \dots}]
[{\large \dots}] ] ]
\path (current bounding box.north west) node[below right,font=\Large] (tl) {$\frac{c}{d}$}
(tl) foreach \x in {2,3,4} {edge (a-\x)}
(current bounding box.north east) node[below right,font=\Large] (tr) {{\large \dots}}
(tr) foreach \x in {2,6,8} {edge (a-\x)}
(a-2) foreach \x in {5,7} {edge (a-\x)};
\draw (tl) edge [very thick] (a-2)
(tl) edge [very thick] (a-3)
(tl) edge [very thick] (a-4)
(a-2) edge [very thick] (a-3)
(a-3) edge [very thick] (a-4);
\end{forest}
\end{center}
\end{subfigure}
\end{floatrow}
\caption*{Fig. 4.7: All four possible orientations of a lineage of order 4.}
\end{figure}
\noindent (In the bottom two trees above, a factor of $q$ is canceled out from their respective numerators and denominators to yield $a/b$ and $c/d$.) Denote the lineage of order $m$ of the $q$-rational at depth $m+j+2$ as $\frac{p_1(q)}{q_1(q)}$ through $\frac{p_{j+3}(q)}{q_{j+3}(q)}$. Using Theorem 4.2, define $C_1$ through $C_{j+2}$ such that
\begin{equation*}
\sum_{i=1}^{j+2}{C_i\frac{q_{i}(1)^{j+1}}{q_{j+3}(1)^{j+1}}\Delta_j\left(\frac{p_i(q)}{q_i{q}}\right)} = \Delta_j\left(\frac{p_{j+3}(q)}{q_{j+3}{q}}\right)
\end{equation*}
in which $\Delta_j([x]_q)$ is defined as it was in said theorem. For all $i \in [1,j+2]$, we have
\begin{equation*}
\left.\frac{d^j}{d^j q}\frac{p_{i}(q)}{q_{i}(q)}\right|_{q=1}-R_{i}(1) = \Delta_j\left(\frac{p_i(q)}{q_i(q)}\right)
\end{equation*}
in which $R_{i}(q)$ depends on $p_{i}(q),p_{i}'(q),\ldots,p_{i}^{(j-1)}(q)$ and $q_{i}(q),q_{i}'(q),\ldots,q_{i}^{(j-1)}(q)$ but not $p_{i}^{(j)}(q)$ nor $q_{i}^{(j)}(q)$. (This can be verified by expanding out the $j$th derivative of $\frac{p_{i}(q)}{q_{i}(q)}$ at $q=1$.) Because we began the proof with the presumption that we have identities for the first through $(j-1)$th derivatives of both the numerators and denominators of the $q$-rationals at $q=1$, each $R_i(1)$ is calculable; moreover, we have
\begin{equation*}
\left.\frac{d^j}{d^j q}\frac{p_{j+3}(q)}{q_{j+3}(q)}\right|_{q=1}=\sum_{i=1}^{j+2}{C_i\frac{q_{i}(1)^{j+1}}{q_{j+3}(1)^{j+1}}\left(\left.\frac{d^j}{d^j q}\frac{p_{i}(q)}{q_{i}(q)}\right|_{q=1}-R_{i}(1)\right)}+R_{j+3}(q)
\end{equation*}
We can calculate the $j$th derivative of $\frac{p_{i}(q)}{q_{i}(q)}$ at $q = 1$ for all $i \in [1,j+2]$ by the inductive hypothesis, and thus we can confirm the identity for the $j$th derivative of $\frac{p_{j+3}(q)}{q_{j+3}(q)}$ at $q = 1$. Repeating this process for all $2^{N+3}$ orientations of $\frac{p_{j+3}(q)}{q_{j+3}(q)}$, we can then complete the proof.
\newpage
\section{\normalfont \textbf{Theorem} \(\mid\) First Derivative of the $q$-Rationals}
Recall that, using the computer experimentations in Section 3, I theorized the following theorem from Section 1.
\begin{theoremoneone}
Let $x$ be a rational number. The first derivative of its  $q$-deformation $[x]_q$ at $q = 1$ is given by
\[
\frac{d}{dq}\left[x\right]_q\Bigr|_{q=1} = \frac{1}{2}\left(x^2 - x + 1 - f(x)^2\right)
\]
in which $f(x)$ is Thomae's function, namely
\[f(x) = \begin{dcases}
      \frac{1}{b} & x = \frac{a}{b},\mathrm{gcd}(a,b)=1 \\
      0 & x \notin \mathbb{Q}
   \end{dcases}
\]
\end{theoremoneone}
\noindent Adam S. Sikora independently formulated that $[x]_q$ has the power series
\[x + \left(x^2-x+\frac{1}{2}-\frac{1}{2m^2}\right)h + O(h^2)\]
for $x = n + 1/m$ and $q = e^h$, in which $n,m \in \mathbb{Z}_{>0}$, in Remark 22 at the bottom of Page 17 of his paper, ``Tangle Equations, the Jones conjecture, and quantum continued fractions" \cite{sikora2020tangle}. This observation boils down to my first derivative identity in the case of $x = n + 1/m$, in which again $n,m \in \mathbb{Z}_{>0}$.
\newline~
\newline
I follow the structure given in Section 4 to prove Theorem 1.1 via induction. One can easily verify that this identity holds for the slab of order $4$ starting at $-1$  on the $q$-deformed Stern-Brocot tree, which we take as a base case. We now assume that Theorem 1.1 holds for every $q$-rational in the slab of order $4$ starting at $i$ on the $q$-deformed Stern-Brocot tree, and denote a given $q$-rational on its last layer $a/b$. We wish to prove that this identity holds for every $q$-rational in the slab of order $4$ starting at $i+1$, which boils down to proving that it holds for any given child of $a/b$. Following the proof structure further, we take $c/d$ to be at depth $i+1$, and break up the alignment of $a/b$, $c/d$, and our child of $a/b$ into four cases.
\newline~
\newline
\underline{Case 1}
\newline
In Cases 1 and 2, $bc - ad = 1$.
\begin{center}
\begin{forest}
for tree={alias/.wrap pgfmath arg={a-#1}{id},font=\Large}
[$\frac{a-q^{n}c}{b-q^{n}d}$
[{\large \dots}
[{\large \dots}] 
[{\large \dots}] ]
[$\frac{a}{b}$
[{\large \dots}]
[$\frac{a+q^{n+1}c}{b+q^{n+1}d}$] ] ]
\path (current bounding box.north west) node[below right,font=\Large] (tl) {{\large \dots}}
(tl) foreach \x in {2,3,4} {edge (a-\x)}
(current bounding box.north east) node[below right,font=\Large] (tr) {$\frac{c}{d}$}
(tr) foreach \x in {2,6,8} {edge (a-\x)}
(a-2) foreach \x in {5,7} {edge (a-\x)};
\draw (tr) edge [very thick] (a-2)
(tr) edge [very thick] (a-6)
(tr) edge [very thick] (a-8)
(a-2) edge [very thick] (a-6)
(a-6) edge [very thick] (a-8);
\end{forest}
\end{center}
We start by taking
\begin{equation*}
\begin{gathered}
\left.\frac{d}{dq}\frac{a+q^{n+1}c}{b+q^{n+1}d}\right|_{q=1}-2\left(\frac{b}{b+d}\right)^2\left.\frac{d}{dq}\frac{a}{b}\right|_{q=1}-2\left(\frac{d}{b+d}\right)^2\left.\frac{d}{dq}\frac{c}{d}\right|_{q=1}+\\
\left(\frac{b-d}{b+d}\right)^2\left.\frac{d}{dq}\frac{a-q^nc}{b-q^nd}\right|_{q=1}=\frac{1}{(b+d)^2}
\end{gathered}
\end{equation*}
Since all vanishing lineages of order 4 were covered in the base case, we know that the lineage of $(a+q^{n+1}c)/(b+q^{n+1}d)$ of order 4 is non-vanishing and may apply Theorem 4.2 to obtain the coefficients $C_1 = 2$, $C_2 = 2$, and $C_3 = -1$. Plugging in the identity for $a/b$, $c/d$, and $(a-c)/(b-d)$, we have
\begin{equation*}
\begin{gathered}
2\left(\frac{b}{b+d}\right)^2\left.\frac{d}{dq}\frac{a}{b}\right|_{q=1}+2\left(\frac{d}{b+d}\right)^2\left.\frac{d}{dq}\frac{c}{d}\right|_{q=1}-\left(\frac{b-d}{b+d}\right)^2\left.\frac{d}{dq}\frac{a-q^nc}{b-q^nd}\right|_{q=1}+\\
\frac{1}{(b+d)^2}=\frac{1}{2}\left(\left(\frac{a+c}{b+d}\right)^2-\frac{a+c}{b+d}+1-\frac{1}{(b+d)^2}\right)=\\
\frac{1}{2}\left(\left(\frac{a+c}{b+d}\right)^2-\frac{a+c}{b+d}+1-f\left(\frac{a+c}{(b+d)}\right)^2\right)
\end{gathered}
\end{equation*}
exactly corresponding to the equation in Theorem 1.1 for $x = (a+c)/(b+d)$. We prove our three remaining cases similarly.
\newline~
\newline
\underline{Case 2}
\newline
Recall that $bc - ad = 1$.
\begin{center}
\begin{forest}
for tree={alias/.wrap pgfmath arg={a-#1}{id},font=\Large}
[$\frac{a-q^{n}c}{b-q^{n}d}$
[{\large \dots}
[{\large \dots}] 
[{\large \dots}] ]
[$\frac{a}{b}$
[$\frac{(q+1)a-q^{n}c}{(q+1)b-q^{n}d}$]
[{\large \dots}] ] ]
\path (current bounding box.north west) node[below right,font=\Large] (tl) {{\large \dots}}
(tl) foreach \x in {2,3,4} {edge (a-\x)}
(current bounding box.north east) node[below right,font=\Large] (tr) {$\frac{c}{d}$}
(tr) foreach \x in {2,6,8} {edge (a-\x)}
(a-2) foreach \x in {5,7} {edge (a-\x)};
\draw (tr) edge [very thick] (a-2)
(tr) edge [very thick] (a-6)
(a-2) edge [very thick] (a-6)
(a-2) edge [very thick] (a-7)
(a-6) edge [very thick] (a-7);
\end{forest}
\end{center}
\begin{equation*}
\begin{gathered}
\left.\frac{d}{dq}\frac{(q+1)a-q^{n}c}{(q+1)b-q^{n}d}\right|_{q=1}-2\left(\frac{b}{b+d}\right)^2\left.\frac{d}{dq}\frac{a}{b}\right|_{q=1}+\left(\frac{d}{b+d}\right)^2\left.\frac{d}{dq}\frac{c}{d}\right|_{q=1}-\\
2\left(\frac{b-d}{b+d}\right)^2\left.\frac{d}{dq}\frac{a-q^nc}{b-q^nd}\right|_{q=1}=\frac{1}{(b+d)^2}
\end{gathered}
\end{equation*}
Plugging in the identity for $a/b$, $c/d$, and $(a-c)/(b-d)$, we have
\begin{equation*}
\begin{gathered}
2\left(\frac{b}{b+d}\right)^2\left.\frac{d}{dq}\frac{a}{b}\right|_{q=1}-\left(\frac{d}{b+d}\right)^2\left.\frac{d}{dq}\frac{c}{d}\right|_{q=1}+2\left(\frac{b-d}{b+d}\right)^2\left.\frac{d}{dq}\frac{a-q^nc}{b-q^nd}\right|_{q=1}+\\
\frac{1}{(b+d)^2}=\frac{1}{2}\left(\left(\frac{2a-c}{2b-d}\right)^2-\frac{2a-c}{2b-d}+1-\frac{1}{(2b-d)^2}\right)=\\
\frac{1}{2}\left(\left(\frac{2a-c}{2b-d}\right)^2-\frac{2a-c}{2b-d}+1-f\left(\frac{2a-c}{2b-d}\right)^2\right)
\end{gathered}
\end{equation*}
exactly corresponding to the equation in Theorem 1.1 for $x = (2a-c)/(2b-d)$.
\newline~
\newline
\underline{Case 3}
\newline
In Cases 3 and 4, $ad - bc = 1$.
\begin{center}
\begin{forest}
for tree={alias/.wrap pgfmath arg={a-#1}{id},font=\Large}
[$\frac{a-c}{b-d}$
[$\frac{a}{b}$
[{\large \dots}] 
[$\frac{(q+1)a-qc}{(q+1)b-qd}$] ]
[{\large \dots}
[{\large \dots}]
[{\large \dots}] ] ]
\path (current bounding box.north west) node[below right,font=\Large] (tl) {$\frac{c}{d}$}
(tl) foreach \x in {2,3,4} {edge (a-\x)}
(current bounding box.north east) node[below right,font=\Large] (tr) {{\large \dots}}
(tr) foreach \x in {2,6,8} {edge (a-\x)}
(a-2) foreach \x in {5,7} {edge (a-\x)};
\draw (tl) edge [very thick] (a-2)
(tl) edge [very thick] (a-3)
(a-2) edge [very thick] (a-3)
(a-2) edge [very thick] (a-5)
(a-3) edge [very thick] (a-5);
\end{forest}
\end{center}
(In the tree above, a factor of $q$ is canceled out from their respective numerators and denominators to yield $a/b$ and $c/d$.)
\begin{equation*}
\begin{gathered}
\left.\frac{d}{dq}\frac{(q+1)a-qc}{(q+1)b-qd}\right|_{q=1}-2\left(\frac{b}{b+d}\right)^2\left.\frac{d}{dq}\frac{a}{b}\right|_{q=1}+\left(\frac{d}{b+d}\right)^2\left.\frac{d}{dq}\frac{c}{d}\right|_{q=1}-\\
2\left(\frac{b-d}{b+d}\right)^2\left.\frac{d}{dq}\frac{a-c}{b-d}\right|_{q=1}=\frac{1}{(b+d)^2}
\end{gathered}
\end{equation*}
Plugging in the identity for $a/b$, $c/d$, and $(a-c)/(b-d)$, we obtain the exact calculation as that which appeared at the end of Case 2.
\newline~
\newline
\underline{Case 4}
\newline
Recall that $ad - bc = 1$.
\begin{center}
\begin{forest}
for tree={alias/.wrap pgfmath arg={a-#1}{id},font=\Large}
[$\frac{a-c}{b-d}$
[$\frac{a}{b}$
[$\frac{qa+c}{qb+d}$] 
[{\large \dots}] ]
[{\large \dots}
[{\large \dots}]
[{\large \dots}] ] ]
\path (current bounding box.north west) node[below right,font=\Large] (tl) {$\frac{c}{d}$}
(tl) foreach \x in {2,3,4} {edge (a-\x)}
(current bounding box.north east) node[below right,font=\Large] (tr) {{\large \dots}}
(tr) foreach \x in {2,6,8} {edge (a-\x)}
(a-2) foreach \x in {5,7} {edge (a-\x)};
\draw (tl) edge [very thick] (a-2)
(tl) edge [very thick] (a-3)
(tl) edge [very thick] (a-4)
(a-2) edge [very thick] (a-3)
(a-3) edge [very thick] (a-4);
\end{forest}
\end{center}
(In the tree above, a factor of $q$ is canceled out from their respective numerators and denominators to yield $a/b$ and $c/d$.)
\begin{equation*}
\begin{gathered}
\left.\frac{d}{dq}\frac{qa+c}{qb+d}\right|_{q=1}-2\left(\frac{b}{b+d}\right)^2\left.\frac{d}{dq}\frac{a}{b}\right|_{q=1}-2\left(\frac{d}{b+d}\right)^2\left.\frac{d}{dq}\frac{c}{d}\right|_{q=1}+\\
\left(\frac{b-d}{b+d}\right)^2\left.\frac{d}{dq}\frac{a-c}{b-d}\right|_{q=1}=\frac{1}{(b+d)^2}
\end{gathered}
\end{equation*}
Plugging in the identity for $a/b$, $c/d$, and $(a-c)/(b-d)$, we obtain the exact calculation as that which appeared at the end of Case 1. This completes the proof.
\section{\normalfont \textbf{Theorem} \(\mid\) Second Derivative of the $q$-Rationals}
Recall that, using the computer experimentations in Section 3, I theorized the following theorem from Section 1.
\begin{theoremonetwo}
Let $\frac{a}{b}$ be an irreducible rational number. The second derivative of its $q$-deformation $\left[\frac{a}{b}\right]_q$ at $q = 1$ is given by
\[
\frac{d^2}{dq^2}\left[\frac{a}{b}\right]_q\Bigr|_{q=1} = F\left(\frac{a}{b}\right) + f\left(\frac{a}{b}\right)^2G\left(\frac{a}{b}\right)+20\sum_{n=1}^{b-1}{\left\langle\frac{n}{a}\right\rangle_ b H\left(\frac{n}{b}\right)}
\]
in which
\begin{itemize}
\item $f(x)$ is Thomae's function (as given in Theorem 1.1)
\item \(F(x) = \frac{1}{3}x^3-x^2+\frac{5}{3}x-1\)
\item \(G(x) = 1 - x\)
\item \(H(x) = x^2(1 - x)\)
\item \(\left\langle\frac{n}{a}\right\rangle_b \coloneqq \frac{na^{-1}\:\mathrm{mod}\:b}{b}-\frac{1}{2}\)
\newline
where $a^{-1}$ is the inverse of $a$ in the multiplicative group \(\left(\mathbb{Z}\slash b\mathbb{Z}\right)^\times\) and \(x\:\mathrm{mod}\:b\) is the integer between $0$ and $b-1$ corresponding to the congruence class of $x$ modulo $b$.
\end{itemize}
\end{theoremonetwo}
\noindent Throughout the proof, when working with fractions denoted by $a/b$, $c/d$, and so forth, I refer to the inverse of $a$ in the multiplicative group \(\left(\mathbb{Z}\slash b\mathbb{Z}\right)^\times\) as $a^{-1}$, the inverse of $c$ in the multiplicative group \(\left(\mathbb{Z}\slash d\mathbb{Z}\right)^\times\) as $c^{-1}$, and so forth.
\newline~
\newline
Recall how in our proof structure, we needed the lower-order derivatives of both the numerators and denominators of the quantum rationals. I will prove an identity for the first derivatives of the numerators of the quantum rationals as a lemma; since we've already proven the first derivatives of the quantum rationals, the first derivatives of their denominators follows directly from this lemma as a corollary.
\begin{lemma}
Given $\left[\frac{a}{b}\right]_q=\frac{a(q)}{b(q)}$, we have that
\[
\frac{d}{dq}a(q)\Bigr|_{q=1}=\frac{1}{2}\left(D\left(\frac{a}{b}\right)b+a^{-1}-a\right)
\]
in which $D(a/b)$ is the depth of $a/b$ on the Stern-Brocot tree and $a^{-1}$ is the inverse of $a$ in the multiplicative group \(\left(\mathbb{Z}\slash b\mathbb{Z}\right)^\times\).
\end{lemma}
\noindent A proof of this lemma is given in Appendix B.
\begin{corollary}
Given $\left[\frac{a}{b}\right]_q=\frac{a(q)}{b(q)}$, we have that
\[
\frac{d}{dq}b(q)\Bigr|_{q=1}=\frac{1}{2a}\left(1-a^2+ba^{-1}+b^2\left(1-D\left(\frac{a}{b}\right)\right)\right)
\]
in which $D(a/b)$ is the depth of $a/b$ on the Stern-Brocot tree and $a^{-1}$ is the inverse of $a$ in the multiplicative group \(\left(\mathbb{Z}\slash b\mathbb{Z}\right)^\times\).
\end{corollary}
\noindent Since we already have a closed-form identity for $\frac{d}{dq}[x]_q$ for all $x \in \mathbb{Q}$, the corollary above follows naturally from the lemma.
\newline~
\newline
Before we prove Theorem 1.2, we also cite two formulas proven by Dohoon Choi, Byungheup Jun, Jungyun Lee, and Subong Lim \cite{choi} that will be instrumental in verifying the values of the second derivatives we encounter in the proof.
\newline~
\newline
Let \[s_{i,j}(a,b)=\sum_{n=1}^{b-1}{\overline{B}_i\left(\frac{n}{b}\right) \overline{B}_j\left(\frac{an}{b}\right)}\]
in which $\overline{B}_i(x)=B_i\left(x-\lfloor x \rfloor\right)$, where $B_i(x)$ is the $i$th Bernoulli polynomial. Choi, Jun, Lee, and Lim \cite{choi} call $s_{i,j}(a,b)$ a \textit{generalized Dedekind sum}.
Note that
\begin{equation*}
\begin{gathered}
\sum_{n=1}^{b-1}{\left\langle\frac{n}{a}\right\rangle_b H\left(\frac{n}{b}\right)}+\sum_{n=1}^{b-1}{\left\langle\frac{n}{a}\right\rangle_b \overline{B}_3\left(\frac{n}{b}\right)}=\\
\frac{1}{2}\sum_{n=1}^{b-1}{\left\langle\frac{n}{a}\right\rangle_b \frac{n}{b}\left(1-\frac{n}{b}\right)} =
\frac{1}{2}\sum_{n=1}^{\left\lfloor b/2 \right\rfloor}{\left(\left\langle\frac{n}{a}\right\rangle_b - \left\langle\frac{b-n}{a}\right\rangle_b\right)\frac{n}{b}\left(1-\frac{n}{b}\right)}=0
\end{gathered}
\end{equation*}
since the term
\[
\left\langle\frac{n}{a}\right\rangle_b - \left\langle\frac{b-n}{a}\right\rangle_b = 0 \quad \forall n \in \left[1,b-1\right],a \in (\mathbb{Z}/b\mathbb{Z})^\times
\]
Of course, $\left\langle\frac{n}{a}\right\rangle_b = \overline{B}_1\left(\frac{an}{b}\right)$, allowing us to make the substitution
\[\sum_{n=1}^{b-1}{\left\langle\frac{n}{a}\right\rangle_b H\left(\frac{n}{b}\right)} = -\sum_{n=0}^{b}{\overline{B}_3\left(\frac{n}{b}\right)\overline{B}_1\left(\frac{an}{b}\right)} = -s_{3,1}\left(a^{-1},b\right)=-s_{1,3}\left(a,b\right)
\]
in the equation in Theorem 1.2. 
\newline~
\newline
Choi, Jun, Lee, and Lim \cite{choi} provide identities governing generalized Dedekind sums that will assist our proof of Theorem 1.2. They define
\begin{equation}
h_{i,j}(a,b)=(-1)^{i+j}\frac{1}{i!j!}(s_{i,j}(a,b)-d_{i,j}B_{i}B_{j})
\tag{6.1}
\end{equation}
in which
 \begin{equation*}
d_{i,j} = \begin{cases}
        1 \text{ if $i=1$ or $j=1$,}
        \\
        0 \text{ otherwise}.
        \end{cases}
 \end{equation*}
and prove the reciprocity formula
\begin{equation}
\begin{gathered}
-q^{i-1}\sum^{j}_{u=0}{\binom{i-1+u}{i-1}h_{i-1+u,b-u}(p,q)(-p)^{u}}+\\
p^{j-1}\sum^{i}_{v=0}{\binom{j-1+v}{j-1}h_{j-1+v,i-v}(q,p)(-q)^{v}}=\\
\frac{B_{i-1}B_{j}}{(i-1)!j!}q(-1)^{j-1}+\frac{B_{i}B_{j-1}}{i!(j-1)!}p(-1)^{j-1}
\tag{6.2}
\end{gathered}
\end{equation}
For even $i$, they show that
\begin{equation}
h_{i,0}(a,b)=h_{0,i}(a,b)=\frac{B_i}{i!}b^{1-i}
\tag{6.3}
\end{equation}
Additionally, they show that
\begin{equation}
h_{i,j}(-p,q)=\begin{cases}
(-1)^{j}h_{i,j}(p,q)-2B_1^2 & \text{if }i=j=1 \\
(-1)^{j}h_{i,j}(p,q) & \text{otherwise}
\end{cases}
\tag{6.4}
\end{equation}
and that
\begin{equation}
h_{i,j}(p+q,q)=h_{i,j}(p,q)
\tag{6.5}
\end{equation}
Finally, let $t_{i,j}(a,b)=p^{i}s_{i,j}(a,b)$. We will need that
\begin{equation}
\begin{gathered}
t_{i,j}(2p,q)=2^{i}t_{i,j}(p,q)\text{ for }2\nmid q \text{ and } i + j = 4 \\
t_{i,j}(p,q/2)=2^{i}t_{i,j}(p,q)\text{ for }2\mid q \text{ and } i + j = 4
\end{gathered}
\tag{6.6}
\end{equation}
which can be verified term by term for each possible choice of $i$ and $j$.
\newpage
\noindent We are now ready to prove the identity via induction by following the structure given in Section 4. One can easily verify that the identity holds for the slab of order $5$ starting at $-1$  on the $q$-deformed Stern-Brocot tree, which we take as a base case. We now assume that Theorem 1.2 holds for every $q$-rational in the slab of order $5$ starting at $i$ on the $q$-deformed Stern-Brocot tree, and denote a given $q$-rational on its last layer $a/b$. We wish to prove that this identity holds for every $q$-rational in the slab of order $5$ starting at $i+1$, which boils down to proving that it holds for any given child of $a/b$. Following the proof structure further, we take $c/d$ to be at depth $i+1$, and break up the alignment of $a/b$, $c/d$, and our child of $a/b$ into eight cases.
\newline~
\newline
\underline{Case 1}
\newline
In Cases 1 and 2, $bc-ad=1$, $a^{-1}=b-d$, and $c^{-1}=b-(n+1)d$.
\begin{center}
\begin{forest}
for tree={alias/.wrap pgfmath arg={a-#1}{id},font=\Large}
[$\frac{a-(q^{n}+q^{n-1})c}{b-(q^{n}+q^{n-1})d}$
[{\large \dots}
[{\large \dots} 
[{\large \dots}] 
[{\large \dots}] ]
[{\large \dots}
[{\large \dots}]
[{\large \dots}] ] ]
[$\frac{a-q^{n}c}{b-q^{n}d}$
[{\large \dots}
[{\large \dots}] 
[{\large \dots}] ] 
[$\frac{a}{b}$ 
[{\large \dots}] 
[$\frac{a+q^{n+1}c}{b+q^{n+1}d}$] ] ] ]
\path (current bounding box.north west) node[below right,font=\Large] (tl) {{\large \dots}}
(tl) foreach \x in {2,3,4,5} {edge (a-\x)}
(current bounding box.north east) node[below right,font=\Large] (tr) {$\frac{c}{d}$}
(tr) foreach \x in {2,10,14,16} {edge (a-\x)}
(a-2) foreach \x in {7,9,11,12} {edge (a-\x)}
(a-3) foreach \x in {6,8} {edge (a-\x)}
(a-10) foreach \x in {13,15} {edge (a-\x)};
\draw (tr) edge [very thick] (a-2)
(tr) edge [very thick] (a-10)
(tr) edge [very thick] (a-14)
(tr) edge [very thick] (a-16)
(a-2) edge [very thick] (a-10)
(a-10) edge [very thick] (a-14)
(a-14) edge [very thick] (a-16);
\end{forest}
\end{center}
We aim to prove that the identity holds for $(a+q^{n+1}c)/(b+q^{n+1}d)$. We start by taking
\begin{equation*}
\begin{gathered}
\left.\frac{d^2}{dq^2}\frac{a+q^{n+1}c}{b+q^{n+1}d}\right|_{q=1}-3\left(\frac{b}{b+d}\right)^3\left.\frac{d^2}{dq^2}\frac{a}{b}\right|_{q=1}-6\left(\frac{d}{b+d}\right)^3\left.\frac{d^2}{dq^2}\frac{c}{d}\right|_{q=1}+\\
3\left(\frac{b-d}{b+d}\right)^3\left.\frac{d^2}{dq^2}\frac{a-q^nc}{b-q^nd}\right|_{q=1}-\left.\frac{d^2}{dq^2}\left(\frac{b-2d}{b+d}\right)^3\frac{a-(q^n+q^{n-1})c}{b-(q^n+q^{n-1})d}\right|_{q=1}=\\
\frac{2 \left(d \left(-3 d a'+a \left(3 d'+d (3 n+2)\right)+3 c
   b'\right)\right)}{(b+d)^3}+\\
   \frac{2\left(-b \left(a d-3 d c'+c \left(6 d'+d (3
   n+2)\right)\right)+b^2 c\right)}{(b+d)^3}=
\frac{6c-4b-4d}{(b+d)^3}
\end{gathered}
\end{equation*}
which we compute using the identity for the first derivative of $\left[\frac{a}{b}\right]_q$ and the lemma. Since all vanishing lineages of order 5 were covered in the base case, we know that the lineage of $(a+q^{n+1}c)/(b+q^{n+1}d)$ of order 5 is non-vanishing and may apply Theorem 4.2 to obtain the coefficients $C_1 = 3$, $C_2 = 6$, $C_3 = -3$, and $C_4 = 1$. Plugging in the identity for $a/b$, $c/d$, $(a-q^{n}c)/(b-q^{n}d)$,\break and $(a-(q^{n}+q^{n=1})c)/(b-(q^{n}+q^{n-1})d)$, we have
\begin{equation*}
\begin{gathered}
3\left(\frac{b}{b+d}\right)^3\left.\frac{d^2}{dq^2}\frac{a}{b}\right|_{q=1}+6\left(\frac{d}{b+d}\right)^3\left.\frac{d^2}{dq^2}\frac{c}{d}\right|_{q=1}-3\left(\frac{b-d}{b+d}\right)^3\left.\frac{d^2}{dq^2}\frac{a-q^nc}{b+q^nd}\right|_{q=1}\\
+\left.\frac{d^2}{dq^2}\left(\frac{b-2d}{b+d}\right)^3\frac{a-(q^n+q^{n-1})c}{b+(q^n+q^{n-1})d}\right|_{q=1}+\frac{6c-4b-4d}{(b+d)^3} = \\
\frac{1}{3}\left(\frac{a+c}{b+d}\right)^3-\left(\frac{a+c}{b+d}\right)^2+\frac{5}{3}\left(\frac{a+c}{b+d}\right)-1+\left(\frac{1}{b+d}\right)^2-
\frac{a+c}{(b+d)^3}+\\
\frac{-4b+2d}{(b+d)^3}-60\left(\frac{b}{b+d}\right)^3s_{1,3}\left(a,b\right)-120\left(\frac{d}{b+d}\right)^3s_{1,3}\left(c,d\right)+\\
60\left(\frac{b-d}{b+d}\right)^3s_{1,3}\left(a-c,b-d\right)-20\left(\frac{b-2d}{b+d}\right)^3s_{1,3}\left(a-2c,b-2d\right)
\end{gathered}
\end{equation*}
Thus, to show that the identity holds for $(a+q^{n}c)/(b+d^{n}d)$, we must show that
\begin{equation*}
\begin{gathered}
20s_{1,3}\left(a+c,b+d\right)-60\left(\frac{b}{b+d}\right)^3s_{1,3}\left(a,b\right)-\\
120\left(\frac{d}{b+d}\right)^3s_{1,3}\left(c,d\right)+60\left(\frac{b-d}{b+d}\right)^3s_{1,3}\left(a-c,b-d\right)-\\
20\left(\frac{b-2d}{b+d}\right)^3s_{1,3}\left(a-2c,b-2d\right)+\frac{-4b+2d}{(b+d)^3}=0
\end{gathered}
\end{equation*}
I denote the equation above as
\begin{equation*}
v_1 \cdot \omega_1 = 0
\end{equation*}
in which
\begin{equation*}
v_1 = \begin{pmatrix}
20(b+d)^3 \\
-60b^3 \\
0 \\
0 \\
-120d^3 \\
60(b-d)^3 \\
-20(b-2d)^3 \\
-4b+2d
\end{pmatrix}
\end{equation*}
and
\begin{equation*}
\omega_1 = \begin{pmatrix}s_{1,3}(a+c,b+d) \\
s_{1,3}(a,b) \\
s_{2,2}(a,b) \\
s_{3,1}(a,b) \\
s_{1,3}(c,d) \\
s_{1,3}(a-c,b-d) \\
s_{1,3}(a-2c,b-2d) \\
1 \end{pmatrix}
\end{equation*}
a notation I will use throughout the rest of the proof for convenience.
Note that $a^{-1}=b-d$, $c^{-1}=b-(n+1)d$, $(a+c)^{-1}=b$, $(a-c)^{-1}=b-2d$, and $(a-2c)^{-1}=b-3d$; taking Equation 6.2 for $i = 4$, $j = 1$, $p = b$, and $q = b + d$, and applying Equations 6.1, 6.3, 6.4, and 6.5, we have
\begin{equation*}
\begin{gathered}
-\frac{1}{6}(b+d)^3s_{3,1}(b,b+d)+\frac{1}{6}b(b+d)^3s_{4,0}(b,b+d)+\frac{1}{24}s_{0,4}(b-d,b)+\\
\frac{1}{6}(b+d)s_{1,3}(b-d,b)+\frac{1}{4}(b+d)^2s_{2,2}(b-d,b)+\\
\frac{1}{6}(b+d)^3s_{3,1}(b-d,b)+\frac{1}{24}(b+d)^4s_{4,0}(b-d,b)=-\frac{1}{720}b\Rightarrow\\
-\frac{1}{6}(b+d)^3s_{3,1}(b,b+d)-\frac{1}{240}b-\frac{1}{720}b^{-3}+\\
\frac{1}{6}(b+d)s_{1,3}(b-d,b)+\frac{1}{4}(b+d)^2s_{2,2}(b-d,b)+\\
\frac{1}{6}(b+d)^3s_{3,1}(b-d,b)-\frac{1}{720}(b+d)^4b^{-3} = 0\Rightarrow
\boxed{u_1 \cdot \omega_1 = 0}
\end{gathered}
\end{equation*}
in which
\begin{equation*}
u_1=\begin{pmatrix}
-(b+d)^{3}/6 \\
(b+d)^3/6 \\
(b+d)^2/4 \\
(b+d)/6 \\
0 \\
0 \\
0 \\
-b/240-b^{-3}/720-b^{-3}(b+d)^4/720
\end{pmatrix}
\end{equation*}
Taking Equation 6.2 for $i = 4$, $j = 1$, $p = b$, and $q = b - d$, and applying Equations 6.1, 6.3, 6.4, and 6.5, we have
\begin{equation*}
\begin{gathered}
\frac{1}{6}(b-d)^3s_{3,1}(b-2d,b-d)+\frac{1}{6}b(b-d)^3s_{4,0}(b-2d,b-d)+\\
\frac{1}{24}s_{0,4}(b-d,b)-\frac{1}{6}(b-d)s_{1,3}(b-d,b)+\frac{1}{4}(b-d)^2s_{2,2}(b-d,b)-\\
\frac{1}{6}(b-d)^3s_{3,1}(b-d,b)+\frac{1}{24}(b-d)^4s_{4,0}(b-d,b)=-\frac{1}{720}b\Rightarrow
\end{gathered}
\end{equation*}
\begin{equation*}
\begin{gathered}
\frac{1}{6}(b-d)^3s_{3,1}(b-2d,b-d)-\frac{1}{240}b-\frac{1}{720}b^{-3}-\\
\frac{1}{6}(b-d)s_{1,3}(b-d,b)+\frac{1}{4}(b-d)^2s_{2,2}(b-d,b)-\\
\frac{1}{6}(b-d)^3s_{3,1}(b-d,b)-\frac{1}{720}(b-d)^4b^{-3}=0\Rightarrow
\boxed{u_2 \cdot \omega_1 = 0}
\end{gathered}
\end{equation*}
in which
\begin{equation*}
u_2=\begin{pmatrix}
0 \\
-(b-d)^3/6 \\
(b-d)^2/4 \\
-(b-d)/6 \\
0 \\
-(b-d)^{3}/6 \\
0 \\
-b/240-b^{-3}/720-b^{-3}(b-d)^4/720
\end{pmatrix}
\end{equation*}
Taking Equation 6.2 for $i = 4$, $j = 1$, $p = b$, and $q = d$, and applying Equations 6.1, 6.3, 6.4, and 6.5, we have
\begin{equation*}
\begin{gathered}
-\frac{1}{6}d^3s_{3,1}(b-(n-1)d,d)+\frac{1}{6}bd^3s_{4,0}(b-(n-1)d,d)+\\
\frac{1}{24}s_{0,4}(b-d,b)+\frac{1}{6}ds_{1,3}(b-d,b)+\frac{1}{4}d^2s_{2,2}(b-d,b)+\\
\frac{1}{6}d^3s_{3,1}(b-d,b)+\frac{1}{24}d^4s_{4,0}(b-d,b)=-\frac{1}{720}b\Rightarrow\\
-\frac{1}{6}d^3s_{3,1}(b-(n-1)d,d)-\frac{1}{240}b-\frac{1}{720}b^{-3}+\\
\frac{1}{6}ds_{1,3}(b-d,b)+\frac{1}{4}d^2s_{2,2}(b-d,b)+\\
\frac{1}{6}d^3s_{3,1}(b-d,b)-\frac{1}{720}d^4b^{-3}=0\Rightarrow
\boxed{u_3 \cdot \omega_1 = 0}
\end{gathered}
\end{equation*}
in which
\begin{equation*}
u_3=\begin{pmatrix}
0 \\
d^3/6 \\
d^2/4 \\
d/6 \\
-d^{3}/6 \\
0 \\
0 \\
-b/240-b^{-3}/720-b^{-3}d^4/720
\end{pmatrix}
\end{equation*}
\newpage
\noindent Finally, taking Equation 6.2 for $i = 4$, $j = 1$, $p = b$, and $q = b - 2d$, and applying Equations 6.1, 6.3, 6.4, 6.5, and 6.6, we have
\begin{equation*}
\begin{gathered}
\frac{1}{3}(b-2d)^3s_{3,1}(b-3d,b-2d)+\frac{1}{6}b(b-2d)^3s_{4,0}(b-3d,b-2d)+\\\frac{2}{3}s_{0,4}(b-d,b)-\frac{4}{3}(b-2d)s_{1,3}(b-d,b)+(b-2d)^2s_{2,2}(b-d,b)-\\\frac{1}{3}(b-2d)^3s_{3,1}(b-d,b)+\frac{1}{24}(b-d)^4s_{4,0}(b-2d,b)=-\frac{1}{720}b\Rightarrow
\end{gathered}
\end{equation*}
\begin{equation*}
\begin{gathered}
\frac{1}{3}(b-2d)^3s_{3,1}(b-3d,b-2d)-\frac{1}{240}b-\frac{1}{45}b^{-3}-\\
\frac{4}{3}(b-2d)s_{1,3}(b-d,b)+(b-2d)^2s_{2,2}(b-d,b)-\\
\frac{1}{3}(b-2d)^3s_{3,1}(b-d,b)-\frac{1}{720}(b-2d)^4b^{-3}=0\Rightarrow
\boxed{u_4 \cdot \omega_1 = 0}
\end{gathered}
\end{equation*}
in which
\begin{equation*}
u_4=\begin{pmatrix}
0 \\
-(b-2d)^3/3 \\
(b-2d)^2 \\
-4(b-2d)/3 \\
0 \\
0 \\
(b-2d)^{3}/3 \\
-b/240-b^{-3}/45-b^{-3}(b-2d)^4/720
\end{pmatrix}
\end{equation*}
Since $v_1=-120u_1+360u_2+720u_3-60u_4$, we see that \[v_1 \cdot \omega_1 = (-120u_1+360u_2+720u_3-60u_4) \cdot \omega_1 = 0\]
and our proof is complete for Case 1.
\newline~
\newline
\underline{Case 2}
\newline
Recall that $bc-ad=1$, $a^{-1}=b-d$, and $c^{-1}=b-(n+1)d$.
\begin{center}
\begin{forest}
for tree={alias/.wrap pgfmath arg={a-#1}{id},font=\Large}
[$\frac{a-(q^{n}+q^{n-1})c}{b-(q^{n}+q^{n-1})d}$
[{\large \dots}
[{\large \dots}
[{\large \dots}] 
[{\large \dots}] ]
[{\large \dots}
[{\large \dots}]
[{\large \dots}] ] ]
[$\frac{a-q^{n}c}{b-q^{n}d}$
[{\large \dots}
[{\large \dots}] 
[{\large \dots}] ] 
[$\frac{a}{b}$ 
[$\frac{(q+1)a-q^{n}c}{(q+1)b-q^{n}d}$] 
[{\large \dots}] ] ] ]
\path (current bounding box.north west) node[below right,font=\Large] (tl) {{\large \dots}}
(tl) foreach \x in {2,3,4,5} {edge (a-\x)}
(current bounding box.north east) node[below right,font=\Large] (tr) {$\frac{c}{d}$}
(tr) foreach \x in {2,10,14,16} {edge (a-\x)}
(a-2) foreach \x in {7,9,11,12} {edge (a-\x)}
(a-3) foreach \x in {6,8} {edge (a-\x)}
(a-10) foreach \x in {13,15} {edge (a-\x)};
\draw (tr) edge [very thick] (a-2)
(tr) edge [very thick] (a-10)
(tr) edge [very thick] (a-14)
(a-2) edge [very thick] (a-10)
(a-10) edge [very thick] (a-14)
(a-10) edge [very thick] (a-15)
(a-14) edge [very thick] (a-15);
\end{forest}
\end{center}
We aim to prove that the identity holds for $((q+1)a-q^{n}c)/((q+1)b-q^{n}d)$. Analogous to Case 1, we start by taking
\begin{equation*}
\begin{gathered}
\left.\frac{d^2}{dq^2}\frac{(q+1)a-q^n c}{(q+1)b-q^n d}\right|_{q=1}-3\left(\frac{b}{2b-d}\right)^3\left.\frac{d^2}{dq^2}\frac{a}{b}\right|_{q=1}+3\left(\frac{d}{2b-d}\right)^3\left.\frac{d^2}{dq^2}\frac{c}{d}\right|_{q=1}-\\
6\left(\frac{b-d}{2b-d}\right)^3\left.\frac{d^2}{dq^2}\frac{a-q^nc}{b-q^nd}\right|_{q=1}+\left.\frac{d^2}{dq^2}\left(\frac{b-2d}{2b-d}\right)^3\frac{a-(q^n+q^{n-1})c}{b-(q^n+q^{n-1})d}\right|_{q=1}=
\end{gathered}
\end{equation*}
\begin{equation*}
\begin{gathered}
-2\left(\frac{b \left(3 d \left(a'+c'\right)+a \left(3 d'+d (3n-2)\right)-c \left(-3 b'+6 d'+3 d n+d\right)\right)}{(2 b-d)^3}\right.+\\
   \left.\frac{d \left(-3
   d a'+a \left(-6 b'+3 d'+3 d n+d\right)+3 c
   b'\right)+b^2 \left(c (2-3 n)-3 c'\right)}{(2 b-d)^3}\right)=\\
   \frac{6a-6c-4b+8d}{(2b-d)^3}
\end{gathered}
\end{equation*}
Following the same steps as before, we must show that
\begin{equation*}
\begin{gathered}
20s_{1,3}\left(2a-c,2b-d\right)-60\left(\frac{b}{b+d}\right)^3s_{1,3}\left(a,b\right)+\\
60\left(\frac{d}{2b-d}\right)^3s_{1,3}\left(c,d\right)-120\left(\frac{b-d}{2b-d}\right)^3s_{1,3}\left(a-c,b-d\right)+\\
20\left(\frac{b-2d}{2b-d}\right)^3s_{1,3}\left(a-2c,b-2d\right)+\frac{2b+2d}{(2b-d)^3}=0
\end{gathered}
\end{equation*}
or, in other words, that
\begin{equation*}
v_2 \cdot \omega_2 = 0
\end{equation*}
in which
\begin{equation*}
v_2 = \begin{pmatrix}
20(2b-d)^3 \\
-60b^3 \\
0 \\
0 \\
60d^3 \\
-120(b-d)^3 \\
20(b-2d)^3 \\
2b+2d
\end{pmatrix}
\end{equation*}
and
\begin{equation*}
\omega_2 = \begin{pmatrix}s_{1,3}(2a-c,2b-d) \\
s_{1,3}(a,b) \\
s_{2,2}(a,b) \\
s_{3,1}(a,b) \\
s_{1,3}(c,d) \\
s_{1,3}(a-c,b-d) \\
s_{1,3}(a-2c,b-2d) \\
1 \end{pmatrix}
\end{equation*}
Note that $a^{-1}=b-d$, $c^{-1}=b-(n+1)d$, $(2a-c)^{-1}=b-d$, $(a-c)^{-1}=b-2d$, and $(a-2c)^{-1}=b-3d$; taking Equation 6.2 for $i = 4$, $j = 1$, $p = b$, and $q = 2b - d$, and applying Equations 6.1, 6.3, 6.4, and 6.5, we have
\begin{equation*}
\begin{gathered}
-\frac{1}{6}(2b-d)^3s_{3,1}(b,2b-d)+\frac{1}{6}b(2b-d)^3s_{4,0}(b,2b-d)+\frac{1}{24}s_{0,4}(b-d,b)-\\
\frac{1}{6}(2b-d)s_{1,3}(b-d,b)+\frac{1}{4}(2b-d)^2s_{2,2}(b-d,b)-\\
\frac{1}{6}(2b-d)^3s_{3,1}(b-d,b)+\frac{1}{24}(2b-d)^4s_{4,0}(b-d,b)=-\frac{1}{720}b\Rightarrow\\
\frac{1}{6}(2b-d)^3s_{3,1}(b-d,2b-d)-\frac{1}{240}b-\frac{1}{720}b^{-3}-\\
\frac{1}{6}(2b-d)s_{1,3}(b-d,b)+\frac{1}{4}(2b-d)^2s_{2,2}(b-d,b)-\\
\frac{1}{6}(2b-d)^3s_{3,1}(b-d,b)-\frac{1}{720}(2b-d)^4b^{-3}=0\Rightarrow\boxed{u_5 \cdot \omega_2=0}
\end{gathered}
\end{equation*}
in which
\begin{equation*}
u_5=\begin{pmatrix}
(2b-d)^3/6 \\
-(2b-d)^3/6 \\
(2b-d)^2/4 \\
-(2b-d)/6 \\
0 \\
0 \\
0 \\
-b/240-b^{-3}/720-b^{-3}(2b-d)^4/720
\end{pmatrix}
\end{equation*}
Since $\omega_1$ and $\omega_2$ only differ in the first row, and the entries in the first rows of $u_2$, $u_3$, and $u_4$ are all $0$, we see that $u_2 \cdot \omega_2 = u_2 \cdot \omega_1 = 0$, $u_3 \cdot \omega_2 = u_3 \cdot \omega_1 = 0$, and $u_4 \cdot \omega_2 = u_4 \cdot \omega_1 = 0$. Furthermore, since $v_2 = -720u_2-360u_3+60u_4+120u_5$, we see that
\[v_2 \cdot \omega_2 = (-720u_2-360u_3+60u_4+120u_5) \cdot \omega_2 = 0\]
and our proof is complete for Case 2. We prove the remaining cases similarly to the first two.
\newpage
\noindent \underline{Case 3}
\newline
In Cases 3 and 4, $bc-ad=2$, $a^{-1}=(b-d)/2$, and $c^{-1}=(b-(2n+3)d)/2$.
\begin{center}
\begin{forest}
for tree={alias/.wrap pgfmath arg={a-#1}{id},font=\Large}
[$\frac{a-q^{n+1}c}{b-q^{n+1}d}$
[{\large \dots}
[{\large \dots} 
[{\large \dots}] 
[{\large \dots}] ]
[{\large \dots}
[{\large \dots}]
[{\large \dots}] ] ]
[$\frac{a+q^{n}c}{b+q^{n}d}$
[$\frac{a}{b}$
[{\large \dots}]
[$\frac{(q^2+q+1)a+q^{n+2}c}{(q^2+q+1)b+q^{n+2}d}$] ] 
[{\large \dots}
[{\large \dots}] 
[{\large \dots}] ] ] ]
\path (current bounding box.north west) node[below right,font=\Large] (tl) {{\large \dots}}
(tl) foreach \x in {2,3,4,5} {edge (a-\x)}
(current bounding box.north east) node[below right,font=\Large] (tr) {$\frac{c}{d}$}
(tr) foreach \x in {2,10,14,16} {edge (a-\x)}
(a-2) foreach \x in {7,9,11,12} {edge (a-\x)}
(a-3) foreach \x in {6,8} {edge (a-\x)}
(a-10) foreach \x in {13,15} {edge (a-\x)};
\draw (tr) edge [very thick] (a-2)
(tr) edge [very thick] (a-10)
(a-2) edge [very thick] (a-10)
(a-2) edge [very thick] (a-11)
(a-10) edge [very thick] (a-11)
(a-11) edge [very thick] (a-13)
(a-10) edge [very thick] (a-13);
\end{forest}
\end{center}
(In the tree above, a factor of $q+1$ is canceled out from their respective numerators and denominators to yield $a/b$ and $c/d$.)
\begin{equation*}
\begin{gathered}
\left.\frac{d^2}{dq^2}\frac{(q^2+q+1)a+q^{n+2}c}{(q^2+q+1)b+q^{n+2}d}\right|_{q=1}-3\left(\frac{2b}{3b+d}\right)^3\left.\frac{d^2}{dq^2}\frac{a}{b}\right|_{q=1}+\\
\left(\frac{2d}{3b+d}\right)^3\left.\frac{d^2}{dq^2}\frac{c}{d}\right|_{q=1}-6\left(\frac{b+d}{3b+d}\right)^3\left.\frac{d^2}{dq^2}\frac{a+q^nc}{b+q^nd}\right|_{q=1}+\\
3\left.\frac{d^2}{dq^2}\left(\frac{b-d}{3b+d}\right)^3\frac{a-q^{n+1}c}{b-q^{n+1}d}\right|_{q=1}=\\
\frac{2 \left(-b \left(6 d \left(a'-c'\right)+a \left(6 d'+d (6
   n-3)\right)+c \left(6 \left(b'+2 d'\right)+d (6
   n+7)\right)\right)\right)}{(3 b+d)^3}+\\
   \frac{2\left(d \left(-6 d a'+a \left(6 \left(2
   b'+d'\right)+d (6 n+7)\right)+6 c b'\right)+b^2 \left(6
   c'+c (6 n-3)\right)\right)}{(3 b+d)^3}=\\
   \frac{3a+3c-6b-2d}{((3 b+d)/2)^3}
\end{gathered}
\end{equation*}
This time, we must have
\begin{equation*}
\begin{gathered}
20s_{1,3}\left(\frac{3a+c}{2},\frac{3b+d}{2}\right)-60\left(\frac{2b}{3b+d}\right)^3s_{1,3}\left(a,b\right)+\\
20\left(\frac{2d}{3b+d}\right)^3s_{1,3}\left(c,d\right)-120\left(\frac{b+d}{3b+d}\right)^3s_{1,3}\left(\frac{a+c}{2},\frac{b+d}{2}\right)+\\60\left(\frac{b-d}{3b+d}\right)^3s_{1,3}\left(\frac{a-c}{2},\frac{b-d}{2}\right)+\frac{-3b+d}{((3b+d)/2)^3}=0
\end{gathered}
\end{equation*}
Define $v_3$ analogously to $v_1$ and $v_2$, and define $\omega_3$ analogously to $\omega_1$ and $\omega_2$:
\newline~
\newline
\begin{tabularx}{\textwidth}{XX}
$v_3 = \begin{pmatrix}
20((3b+d)/2)^3 \\
-60b^3 \\
0 \\
0 \\
20d^3 \\
-120((b+d)/2)^3 \\
60((b-d)/2)^3 \\
-3b+d
\end{pmatrix}$ & $\omega_3 = \begin{pmatrix}
s_{1,3}((3a+c)/2,(3b+d)/2) \\
s_{1,3}(a,b) \\
s_{2,2}(a,b) \\
s_{3,1}(a,b) \\
s_{1,3}{c,d} \\
s_{1,3}((a+c)/2,(b+d)/2) \\
s_{1,3}((a-c)/2,(b-d)/2) \\
1
\end{pmatrix}$
\end{tabularx}
\newline~
\newline
Observe that $a^{-1}=(b-d)/2$, $c^{-1}=(b-(2n+3)d)/2$, $((3b+d)/2)^{-1}=b$, $((a+c)/2)^{-1}=(b-d)/2$, and $((a-c)/2)^{-1}=(b-3d)/2$. Using these identities, derive $u_6$ from Equation 6.2 for $i = 4$, $j = 1$, $p = b$, and $q = (3b+d)/2$ analogously to how we derived $u_1$ through $u_5$ from Equation 6.2 in Cases 1 and 2. Furthermore, derive $u_7$ likewise for $q = (b+d)/2$, $u_8$ likewise for $q = d$, and $u_9$ likewise for $q = (b-d)/2$. This process yields the following vectors:
\begin{equation*}
\begin{gathered}
 u_6 = \begin{pmatrix}
-((3b+d)/2)^3/6 \\
((3b+d)/2)^3/6 \\
((3b+d)/2)^2/4 \\
((3b+d)/2)/6 \\
0 \\
0 \\
0 \\
-b/240-b^{-3}/720-b^{-3}((3b+d)/2)^4/720
\end{pmatrix} \\
u_7 = \begin{pmatrix}
0 \\
((b+d)/2)^3/6 \\
((b+d)/2)^2/4 \\
((b+d)/2)/6 \\
0 \\
-((b+d)/2)^3/6 \\
0 \\
-b/240-b^{-3}/720-b^{-3}((b+d)/2)^4/720
\end{pmatrix} \\
u_8 = \begin{pmatrix}
0 \\
-d^3/3 \\
d^2 \\
-4d/3 \\
-d^3/3 \\
0 \\
0 \\
-b/240-b^{-3}/45-b^{-3}d^4/720
\end{pmatrix}
\end{gathered}
\end{equation*}
\begin{equation*}
u_9 = \begin{pmatrix}
0 \\
-((b-d)/2)^3/6 \\
((b-d)/2)^2/4 \\
-((b-d)/2)/6 \\
0 \\
0 \\
((b-d)/2)^3/6 \\
-b/240-b^{-3}/720-b^{-3}((b-d)/2)^4/720
\end{pmatrix}
\end{equation*}
\newline~
\newline
We see that
$v_3=-120u_6+720u_7-60u_8+360u_9$, and thus
\[v_3 \cdot \omega_3 = (-120u_6+720u_7-60u_8+360u_9) \cdot \omega_3 = 0\]
as desired.
\newline~
\newline
\underline{Case 4}
\newline
Recall that $bc-ad=2$, $a^{-1}=(b-d)/2$, and $c^{-1}=(b-(2n+3)d)/2$.
\begin{center}
\begin{forest}
for tree={alias/.wrap pgfmath arg={a-#1}{id},font=\Large}
[$\frac{a-q^{n+1}c}{b-q^{n+1}d}$
[{\large \dots}
[{\large \dots} 
[{\large \dots}] 
[{\large \dots}] ]
[{\large \dots}
[{\large \dots}]
[{\large \dots}] ] ]
[$\frac{a+q^{n}c}{b+q^{n}d}$
[$\frac{a}{b}$
[$\frac{(q^2+q+1)a-q^{n+1}c}{(q^2+q+1)b-q^{n+1}d}$]
[{\large \dots}] ] 
[{\large \dots}
[{\large \dots}] 
[{\large \dots}] ] ] ]
\path (current bounding box.north west) node[below right,font=\Large] (tl) {{\large \dots}}
(tl) foreach \x in {2,3,4,5} {edge (a-\x)}
(current bounding box.north east) node[below right,font=\Large] (tr) {$\frac{c}{d}$}
(tr) foreach \x in {2,10,14,16} {edge (a-\x)}
(a-2) foreach \x in {7,9,11,12} {edge (a-\x)}
(a-3) foreach \x in {6,8} {edge (a-\x)}
(a-10) foreach \x in {13,15} {edge (a-\x)};
\draw (tr) edge [very thick] (a-2)
(tr) edge [very thick] (a-10)
(a-2) edge [very thick] (a-10)
(a-2) edge [very thick] (a-11)
(a-2) edge [very thick] (a-12)
(a-10) edge [very thick] (a-11)
(a-11) edge [very thick] (a-12);
\end{forest}
\end{center}
(In the tree above, a factor of $q+1$ is canceled out from their respective numerators and denominators to yield $a/b$ and $c/d$.)
\begin{equation*}
\begin{gathered}
\left.\frac{d^2}{dq^2}\frac{(q^2+q+1)a-q^{n+1}c}{(q^2+q+1)b-q^{n+1}d}\right|_{q=1}-3\left(\frac{2b}{3b-d}\right)^3\left.\frac{d^2}{dq^2}\frac{a}{b}\right|_{q=1}-\\
\left(\frac{2d}{3b-d}\right)^3\left.\frac{d^2}{dq^2}\frac{c}{d}\right|_{q=1}+3\left(\frac{b+d}{3b-d}\right)^3\left.\frac{d^2}{dq^2}\frac{a+q^nc}{b+q^nd}\right|_{q=1}-\\
6\left.\frac{d^2}{dq^2}\left(\frac{b-d}{3b-d}\right)^3\frac{a-q^{n+1}c}{b-q^{n+1}d}\right|_{q=1}=
\end{gathered}
\end{equation*}
\begin{equation*}
\begin{gathered}
-\frac{2 \left(b \left(6 d \left(a'+c'\right)+3 a \left(2 d'+2 d
   n+d\right)-c \left(-6 b'+12 d'+6 d n+5 d\right)\right)\right)}{(3
   b-d)^3}-\\
   \frac{2\left(d
   \left(-6 d a'+a \left(6 \left(d'-2 b'\right)+d (6 n+5)\right)+6
   c b'\right)-3 b^2 \left(2 c'+2 c n+c\right)\right)}{(3
   b-d)^3}=\\
   \frac{3a-3c+4d}{((3b-d)/2)^3}
\end{gathered}
\end{equation*}
We must have
\begin{equation*}
\begin{gathered}
20s_{1,3}\left(\frac{3a-c}{2},\frac{3b-d}{2}\right)-60\left(\frac{2b}{3b-d}\right)^3s_{1,3}\left(a,b\right)-\\
20\left(\frac{2d}{3b-d}\right)^3s_{1,3}\left(c,d\right)+60\left(\frac{b+d}{3b-d}\right)^3s_{1,3}\left(\frac{a+c}{2},\frac{b+d}{2}\right)-\\
120\left(\frac{b-d}{3b-d}\right)^3s_{1,3}\left(\frac{a-c}{2},\frac{b-d}{2}\right)+\frac{3b+d}{((3b-d)/2)^3}=0
\end{gathered}
\end{equation*}
Define $\omega_4$ analogously to $\omega_1$, $\omega_2$, and $\omega_3$, and define $v_4$ analogously to $v_1$, $v_2$, and $v_3$. Derive $u_{10}$ from Equation 6.2 for $i = 4$, $j = 1$, $p = b$, and $q = (3b-d)/2$; we see that $v_4=-360u_7+60u_8+720u_9+120u_{10}$, and thus
\[v_4 \cdot \omega_4 = (-360u_7+60u_8-720u_9+120u_{10}) \cdot \omega_4 = 0\]
as desired.
\newline~
\newline
\underline{Case 5}
\newline
In Cases 5 and 6, $ad - bc = 2$, $a^{-1}=(b+d)/2$, and $c^{-1}=(5d-b)/2$.
\begin{center}
\begin{forest}
for tree={alias/.wrap pgfmath arg={a-#1}{id},font=\Large}
[$\frac{a-c}{b-d}$
[$\frac{a+qc}{b+qd}$
[{\large \dots} 
[{\large \dots}] 
[{\large \dots}] ]
[$\frac{a}{b}$
[{\large \dots}]
[$\frac{(q^2+q+1)a-q^2c}{(q^2+q+1)b-q^2d}$] ] ]
[{\large \dots}
[{\large \dots}
[{\large \dots}]
[{\large \dots}] ] 
[{\large \dots}
[{\large \dots}] 
[{\large \dots}] ] ] ]
\path (current bounding box.north west) node[below right,font=\Large] (tl) {$\frac{c}{d}$}
(tl) foreach \x in {2,3,4,5} {edge (a-\x)}
(current bounding box.north east) node[below right,font=\Large] (tr) {{\large \dots}}
(tr) foreach \x in {2,10,14,16} {edge (a-\x)}
(a-2) foreach \x in {7,9,11,12} {edge (a-\x)}
(a-3) foreach \x in {6,8} {edge (a-\x)}
(a-10) foreach \x in {13,15} {edge (a-\x)};
\draw (tl) edge [very thick] (a-2)
(tl) edge [very thick] (a-3)
(a-2) edge [very thick] (a-3)
(a-2) edge [very thick] (a-7)
(a-2) edge [very thick] (a-9)
(a-3) edge [very thick] (a-7)
(a-7) edge [very thick] (a-9);
\end{forest}
\end{center}
\newpage
\noindent (In the tree above, a factor of $q^2+q$ is canceled out from their respective numerators and denominators to yield $a/b$ and $c/d$.)
\begin{equation*}
\begin{gathered}
\left.\frac{d^2}{dq^2}\frac{(q^2+q+1)a-q^2c}{(q^2+q+1)b-q^2d}\right|_{q=1}-3\left(\frac{2b}{3b-d}\right)^3\left.\frac{d^2}{dq^2}\frac{a}{b}\right|_{q=1}-\\
\left(\frac{2d}{3b-d}\right)^3\left.\frac{d^2}{dq^2}\frac{c}{d}\right|_{q=1}+3\left(\frac{b+d}{3b-d}\right)^3\left.\frac{d^2}{dq^2}\frac{a+qc}{b+qd}\right|_{q=1}-\\
6\left.\frac{d^2}{dq^2}\left(\frac{b-d}{3b-d}\right)^3\frac{a-c}{b-d}\right|_{q=1}=\\
\frac{2 \left(b \left(6 d \left(a'+c'\right)-3 a \left(d-2
   d'\right)+c \left(6 b'-12 d'-7 d\right)\right)\right)}{(3 b-d)^3}+\\
   \frac{2\left(d \left(-6 d
   a'+a \left(6 \left(d'-2 b'\right)+7 d\right)+6 c b'\right)+3
   b^2 \left(c-2 c'\right)\right)}{(3 b-d)^3}=\\
   \frac{3a-3c-6b+2d}{((3 b-d)/2)^3}
\end{gathered}
\end{equation*}
We must have
\begin{equation*}
\begin{gathered}
20s_{1,3}\left(\frac{3a-c}{2},\frac{3b-d}{2}\right)-60\left(\frac{2b}{3b-d}\right)^3s_{1,3}\left(a,b\right)-\\
20\left(\frac{2d}{3b-d}\right)^3s_{1,3}\left(c,d\right)+60\left(\frac{b+d}{3b-d}\right)^3s_{1,3}\left(\frac{a+c}{2},\frac{b+d}{2}\right)-\\
120\left(\frac{b-d}{3b-d}\right)^3s_{1,3}\left(\frac{b-d}{2},\frac{3b+d}{2}\right)+\frac{-3b-d}{((3b-d)/2)^3}=0
\end{gathered}
\end{equation*}
It's worth noting that this does \textit{not} contradict our result in Case 4, as $a$ and $c$ are different than before. Denote our new choices of $a$ and $c$ as $\hat{a}$ and $\hat{c}$,; by the reflexivity of the Stern-Brocot tree, we see that $\hat{a} = b - a$ and $\hat{c} = d - c$, in which $a$, $b$, $c$, and $d$ are the same as they were in Case 4. From Case 4, we have
\begin{equation*}
\begin{gathered}
20s_{1,3}\left(\frac{3a-c}{2},\frac{3b-d}{2}\right)-60\left(\frac{2b}{3b-d}\right)^3s_{1,3}\left(a,b\right)-\\
20\left(\frac{2d}{3b-d}\right)^3s_{1,3}\left(c,d\right)+60\left(\frac{b+d}{3b-d}\right)^3s_{1,3}\left(\frac{a+c}{2},\frac{b+d}{2}\right)-\\
120\left(\frac{b-d}{3b-d}\right)^3s_{1,3}\left(\frac{a-c}{2},\frac{b-d}{2}\right)+\frac{3b+d}{((3b-d)/2)^3}=0
\end{gathered}
\end{equation*}
By applying the identity $s_{1,3}(q-p,q)=-s_{1,3}(p,q)$ (by Equations 6.4 and 6.5), we can see that our result in Case 4 implies our desired result in Case 5.
\newline~
\newline
\underline{Case 6}
\newline
Recall that $ad - bc = 2$, $a^{-1}=(b+d)/2$, and $c^{-1}=(5d-b)/2$.
\begin{center}
\begin{forest}
for tree={alias/.wrap pgfmath arg={a-#1}{id},font=\Large}
[$\frac{a-c}{b-d}$
[$\frac{a+qc}{b+qd}$
[{\large \dots} 
[{\large \dots}] 
[{\large \dots}] ]
[$\frac{a}{b}$
[$\frac{(q^2+q+1)a+qc}{(q^2+q+1)b+qd}$]
[{\large \dots}] ] ]
[{\large \dots}
[{\large \dots}
[{\large \dots}]
[{\large \dots}] ] 
[{\large \dots}
[{\large \dots}] 
[{\large \dots}] ] ] ]
\path (current bounding box.north west) node[below right,font=\Large] (tl) {$\frac{c}{d}$}
(tl) foreach \x in {2,3,4,5} {edge (a-\x)}
(current bounding box.north east) node[below right,font=\Large] (tr) {{\large \dots}}
(tr) foreach \x in {2,10,14,16} {edge (a-\x)}
(a-2) foreach \x in {7,9,11,12} {edge (a-\x)}
(a-3) foreach \x in {6,8} {edge (a-\x)}
(a-10) foreach \x in {13,15} {edge (a-\x)};
\draw (tl) edge [very thick] (a-2)
(tl) edge [very thick] (a-3)
(a-2) edge [very thick] (a-3)
(a-2) edge [very thick] (a-7)
(a-3) edge [very thick] (a-7)
(a-3) edge [very thick] (a-8)
(a-7) edge [very thick] (a-8);
\end{forest}
\end{center}
(In the tree above, a factor of $q^2+q$ is canceled out from their respective numerators and denominators to yield $a/b$ and $c/d$.)
\begin{equation*}
\begin{gathered}
\left.\frac{d^2}{dq^2}\frac{(q^2+q+1)a+qc}{(q^2+q+1)b+qd}\right|_{q=1}-3\left(\frac{2b}{3b+d}\right)^3\left.\frac{d^2}{dq^2}\frac{a}{b}\right|_{q=1}+\\
\left(\frac{2d}{3b+d}\right)^3\left.\frac{d^2}{dq^2}\frac{c}{d}\right|_{q=1}-6\left(\frac{b+d}{3b+d}\right)^3\left.\frac{d^2}{dq^2}\frac{a+qc}{b+qd}\right|_{q=1}+\\
3\left.\frac{d^2}{dq^2}\left(\frac{b-d}{3b+d}\right)^3\frac{a-c}{b-d}\right|_{q=1}=\\
\frac{2 b \left(6 d \left(a'-c'\right)+3 a \left(2
   d'+d\right)+c \left(6 b'+12 d'+5 d\right)\right)}{(3 b+d)^3}-\\
   \frac{2 d \left(-6
   d a'+a \left(6 \left(2 b'+d'\right)+5 d\right)+6 c
   b'\right)+6 b^2 \left(2 c'+c\right)}{(3 b+d)^3}=\frac{3a+3c-4d}{(3b+d)^3}
\end{gathered}
\end{equation*}
We must have
\begin{equation*}
\begin{gathered}
20s_{1,3}\left(\frac{3a+c}{2},\frac{3b+d}{2}\right)-60\left(\frac{2b}{3b+d}\right)^3s_{1,3}\left(a,b\right)+\\
20\left(\frac{2d}{3b+d}\right)^3s_{1,3}\left(c,d\right)-120\left(\frac{b+d}{3b+d}\right)^3s_{1,3}\left(\frac{a+c}{2},\frac{b+d}{2}\right)+\\
60\left(\frac{b-d}{3b+d}\right)^3s_{1,3}\left(\frac{a-c}{2},\frac{b-d}{2}\right)+\frac{3b-d}{((3b+d)/2)^3}=0
\end{gathered}
\end{equation*}
which follows from Case 3 analogously to how Case 5 follows from Case 4.
\newline~
\newline
\underline{Case 7}
\newline
In Cases 7 and 8, $ad - bc = 1$, $a^{-1}=d$, and $c^{-1}=4d-b$.
\begin{center}
\begin{forest}
for tree={alias/.wrap pgfmath arg={a-#1}{id},font=\Large}
[$\frac{a-(q+1)c}{b-(q+1)d}$
[$\frac{a-c}{b-d}$
[$\frac{a}{b}$ 
[{\large \dots}] 
[$\frac{(q+1)a-qc}{(q+1)b-qd}$] ]
[{\large \dots}
[{\large \dots}]
[{\large \dots}] ] ]
[{\large \dots}
[{\large \dots}
[{\large \dots}]
[{\large \dots}] ] 
[{\large \dots}
[{\large \dots}] 
[{\large \dots}] ] ] ]
\path (current bounding box.north west) node[below right,font=\Large] (tl) {$\frac{c}{d}$}
(tl) foreach \x in {2,3,4,5} {edge (a-\x)}
(current bounding box.north east) node[below right,font=\Large] (tr) {{\large \dots}}
(tr) foreach \x in {2,10,14,16} {edge (a-\x)}
(a-2) foreach \x in {7,9,11,12} {edge (a-\x)}
(a-3) foreach \x in {6,8} {edge (a-\x)}
(a-10) foreach \x in {13,15} {edge (a-\x)};
\draw (tl) edge [very thick] (a-2)
(tl) edge [very thick] (a-3)
(tl) edge [very thick] (a-4)
(a-2) edge [very thick] (a-3)
(a-3) edge [very thick] (a-4)
(a-3) edge [very thick] (a-6)
(a-4) edge [very thick] (a-6);
\end{forest}
\end{center}
(In the tree above, a factor of $q^2$ is canceled out from their respective numerators and denominators to yield $a/b$ and $c/d$.)
\begin{equation*}
\begin{gathered}
\left.\frac{d^2}{dq^2}\frac{(q+1)a-qc}{(q+1)b-qd}\right|_{q=1}-3\left(\frac{b}{2b-d}\right)^3\left.\frac{d^2}{dq^2}\frac{a}{b}\right|_{q=1}+3\left(\frac{d}{2b-d}\right)^3\left.\frac{d^2}{dq^2}\frac{c}{d}\right|_{q=1}-\\
6\left(\frac{b-d}{2b-d}\right)^3\left.\frac{d^2}{dq^2}\frac{a-c}{b-d}\right|_{q=1}+\left.\frac{d^2}{dq^2}\left(\frac{b-2d}{2b-d}\right)^3\frac{a-(q+1)c}{b-(q+1)d}\right|_{q=1}=\\
\frac{2 \left(b \left(3 d \left(a'+c'\right)-a \left(d-3
   d'\right)+c \left(3 b'-6 d'-2 d\right)\right)\right)}{(2 b-d)^3}+\\
   \frac{2\left(d \left(-3 d
   a'+a \left(-6 b'+3 d'+2 d\right)+3 c b'\right)+b^2
   \left(c-3 c'\right)\right)}{(2 b-d)^3}=\\
\frac{6a-6c-8b+4d}{(2 b-d)^3}
\end{gathered}
\end{equation*}
We must have
\begin{equation*}
\begin{gathered}
20s_{1,3}\left(2a-c,2b-d\right)-60\left(\frac{b}{b+d}\right)^3s_{1,3}\left(a,b\right)+\\
60\left(\frac{d}{2b-d}\right)^3s_{1,3}\left(c,d\right)-120\left(\frac{b-d}{2b-d}\right)^3s_{1,3}\left(a-c,b-d\right)+\\
20\left(\frac{b-2d}{2b-d}\right)^3s_{1,3}\left(a-2c,b-2d\right)+\frac{-2b-2d}{(2b-d)^3}=0
\end{gathered}
\end{equation*}
which follows from Case 2 analogously to how Cases 5 and 6 follow from Cases 4 and 3, respectively.
\newline~
\newline
\underline{Case 8}
\newline
Recall that $ad - bc = 1$, $a^{-1}=d$, and $c^{-1}=4d-b$.
\begin{center}
\begin{forest}
for tree={alias/.wrap pgfmath arg={a-#1}{id},font=\Large}
[$\frac{a-(q+1)c}{b-(q+1)d}$
[$\frac{a-c}{b-d}$
[$\frac{a}{b}$ 
[$\frac{qa+c}{qb+d}$] 
[{\large \dots}] ]
[{\large \dots}
[{\large \dots}]
[{\large \dots}] ] ]
[{\large \dots}
[{\large \dots}
[{\large \dots}]
[{\large \dots}] ] 
[{\large \dots}
[{\large \dots}] 
[{\large \dots}] ] ] ]
\path (current bounding box.north west) node[below right,font=\Large] (tl) {$\frac{c}{d}$}
(tl) foreach \x in {2,3,4,5} {edge (a-\x)}
(current bounding box.north east) node[below right,font=\Large] (tr) {{\large \dots}}
(tr) foreach \x in {2,10,14,16} {edge (a-\x)}
(a-2) foreach \x in {7,9,11,12} {edge (a-\x)}
(a-3) foreach \x in {6,8} {edge (a-\x)}
(a-10) foreach \x in {13,15} {edge (a-\x)};
\draw (tl) edge [very thick] (a-2)
(tl) edge [very thick] (a-3)
(tl) edge [very thick] (a-4)
(tl) edge [very thick] (a-5)
(a-2) edge [very thick] (a-3)
(a-3) edge [very thick] (a-4)
(a-4) edge [very thick] (a-5);
\end{forest}
\end{center}
(In the tree above, a factor of $q^2$ is canceled out from their respective numerators and denominators to yield $a/b$ and $c/d$.)
\begin{equation*}
\begin{gathered}
\left.\frac{d^2}{dq^2}\frac{a+qc}{b+qd}\right|_{q=1}-3\left(\frac{b}{b+d}\right)^3\left.\frac{d^2}{dq^2}\frac{a}{b}\right|_{q=1}-6\left(\frac{d}{b+d}\right)^3\left.\frac{d^2}{dq^2}\frac{c}{d}\right|_{q=1}+\\
3\left(\frac{b-d}{b+d}\right)^3\left.\frac{d^2}{dq^2}\frac{a-c}{b-d}\right|_{q=1}-\left.\frac{d^2}{dq^2}\left(\frac{b-2d}{b+d}\right)^3\frac{a-(q+1)c}{b-(q+1)d}\right|_{q=1}=\\
-\frac{2 \left(b \left(d \left(2 a'+c'\right)+2 a d'+2 c
   \left(b'-d'\right)\right)+d \left(-d a'+a \left(d'-4
   b'\right)+c b'\right)-2 b^2 c'\right)}{(b+d)^3}\\
=\frac{6c+4b-8d}{(b+d)^3}
\end{gathered}
\end{equation*}
We must have
\begin{equation*}
\begin{gathered}
20s_{1,3}\left(a+c,b+d\right)-60\left(\frac{b}{b+d}\right)^3s_{1,3}\left(a,b\right)-\\
120\left(\frac{d}{b+d}\right)^3s_{1,3}\left(c,d\right)+60\left(\frac{b-d}{b+d}\right)^3s_{1,3}\left(a-c,b-d\right)-\\
20\left(\frac{b-2d}{b+d}\right)^3s_{1,3}\left(a-2c,b-2d\right)+\frac{4b-2d}{(b+d)^3}=0
\end{gathered}
\end{equation*}
which follows from Case 1 analogously to how Case 5, 6, and 7 follow from Case 4, 3, and 2, respectively. This completes the proof.
\section{\normalfont \textbf{Conclusion} \(\mid\) Possibilities for Future Research}
My two identities and their proofs shed some important insights into what future terms of the power series of $[a/b]_q$ may look like and how we can prove them. It is reasonable to speculate that the $i$th derivative of $[a/b]_q$ at $q = 1$ is of the form $\eta(a,b)/b^{i+1} + K \times s_{1,i+1}(a,b)$, in which  $\eta(a,b)$ is a polynomial of order $i + 1$ with integer coefficients and the formal parameters $a$ and $b$, $K \in \mathbb{Z}$, and $s_{1,i+1}(a,b)$ is a generalized Dedekind sum as defined by Choi, Jun, Lee, and Lim \cite{choi}. Since $b^{i+1}s_{1,i+1}(a,b) \in \mathbb{Z}$ \cite{apostol}, this form yields an integer when multiplied by $b^{i+1}$. The correct form must have this property, as one can see by expanding the $i$th derivative of $[a/b]_q$ at $q = 1$. Additionally, an identity for $i$th derivative of $[a/b]_q$ at $q = 1$ could be proven using the proof structure I gave in Section 4: split the inductive hypothesis into $2^{i+1}$ cases, and then use the coefficients given in Theorem 4.2 to break down each case. If the form given above is correct, each case will be trivial if $i$ is odd, and may be broken down with Equation 6.2 if $i$ is even. Of course, the number of cases in each proof grows exponentially with $i$, so more research is necessary to determine whether these cases can be proven concurrently. Hopefully, a more involved proof structure can generalize these cases and shed light on the form of the $i$th derivative of $[a/b]_q$ at $q = 1$ in more detail than my speculation above.
\section*{\normalfont \textbf{Acknowledgements}}
I would like to thank Prof. Richard Schwartz for introducing me to the concept of $q$-rational numbers, discussing my ideas with me throughout the research process, and helping me organize and write this paper. Additionally, I would like to thank Prof. Sophie Morier-Genoud and Prof. Valentin Ovsienko for talking with me about my results and opening up avenues for eventual publication and further research. I would like to thank my editor, David Hernandez, for his thorough comments on my paper and his recommendation of it for publication. I would also like to thank my parents for helping me edit this paper and offering useful feedback about how to communicate certain ideas to the reader. Finally, I would like to thank my friends Amelia, Max, Luke, Noah, and others for our conversations about my research and for the support of and investment in this paper that they have continued to show.
\section*{\normalfont \textbf{Appendix A} \(\mid\) Theorem from Section 2}
We aim to prove that $\alpha(q)/\beta(q) = [\alpha(1)/\beta(1)]_q$ for all relatively prime\break$\alpha(1) \in \mathbb{Z}, \beta(1) \in \mathbb{Z} \setminus 0$ via induction. Before diving into the proof, we define the following terminology.
\begin{definitionaone}
Define $\alpha/\beta$ and $\gamma/\delta$ such that $\lvert \alpha \delta - \beta \gamma\rvert = 1$. If both $\alpha/\beta$ and $\gamma/\delta$ are integers, we call their Farey sum a \textit{left branch}. Otherwise, let $\alpha/\beta$ be of the greater depth on the Stern-Brocot tree; we again call the Farey sum of $\alpha/\beta$ and $\gamma/\delta$ a left branch if $\alpha/\beta > \gamma/\delta$, or a \textit{right branch} if $\alpha/\beta < \gamma/\delta$. 
\end{definitionaone}
\newpage
\noindent We also observe that the continued fraction expansion of $u/v$ informs its placement along the Stern-Brocot tree. Let
\[
\frac{\alpha}{\beta} = u_0 + \cfrac{1}{u_1+\cfrac{1}{\cdots + \cfrac{1}{u_n+1}}}
\]
Starting the Farey tree at $u_0$ and $u_0 + 1$, take $u_1$ left branches, $u_2$ right branches, $u_3$ left branches, and so on down the tree through $u_n$; you will arrive at $u/v$.
\newline~
\newline
We are now ready to begin the proof. Given any $c/d$ on the Stern-Brocot tree, it has two parents, one of which has depth one less than it has on the tree. Refer to this parent as $a/b$, and define its continued fractions as
\[
\frac{a}{b}=a_0+\cfrac{1}{a_1+\cfrac{1}{\cdots+\cfrac{1}{a_m}}}\]
We must have that
\[
\frac{c}{d}=a_0+\cfrac{1}{a_1+\cfrac{1}{\cdots+\cfrac{1}{a_{m-1}}}}\]
and
\[
\frac{a}{b} \oplus \frac{c}{d} = a_0+\cfrac{1}{a_1+\cfrac{1}{\cdots+\cfrac{1}{a_m+1}}}\]
\cite{Shrader_Frechette}.
Our base cases (when $\alpha(1)/\beta(1) \in \mathbb{Z}$) are true by definition. Suppose that
\[
\frac{a(q)}{b(q)}=\left[\frac{a(1)}{b(1)}\right]_q=[a_0]_q+\cfrac{q^{a_0}}{[a_1]_{1/q}+\cfrac{q^{-a_1}}{\cdots+\cfrac{q^{a_{m-1}}}{\left[a_m\right]_{1/q}}}}
\]
and
\[
\frac{c(q)}{d(q)}=\left[\frac{c(1)}{d(1)}\right]_q=[a_0]_q+\cfrac{q^{a_0}}{[a_1]_{1/q}+\cfrac{q^{-a_1}}{\cdots+\cfrac{q^{-a_{m-2}}}{\left[a_{m-1}\right]_q}}}
\]
From the relationship between a rational number's continued fraction expansion and its placement on the Stern-Brocot tree, we know that $m$ must be odd. Let
\[
\frac{P(q)}{Q(q)}=[a_0]_q+\cfrac{q^{a_0}}{[a_1]_{1/q}+\cfrac{q^{-a_1}}{\cdots+\cfrac{q^{a_{m-3}}}{\left[a_{m-2}\right]_{1/q}}}}
\]
By the recurrence relation governing generalized continued fractions \cite{jones}, we find that
\[
\frac{a(q)}{b(q)}=\frac{\left[a_m\right]_{q}c(q)+q^{a_{m-1}+a_m-1}P(q)}{\left[a_m\right]_{q}d(q)+q^{a_{m-1}+a_m-1}Q(q)}
\]
and obtain
\begin{equation*}
\begin{gathered}
\left[\frac{a(1)+c(1)}{b(1)+d(1)}\right]_q=[a_0]_q+\cfrac{q^{a_0}}{[a_1]_{1/q}+\cfrac{q^{-a_1}}{\cdots+\cfrac{q^{a_{m-1}}}{\left[a_m+1\right]_{1/q}}}}=\\
\frac{\left[a_m+1\right]_{q}c(q)+q^{a_{m-1}+a_m}P(q)}{\left[a_m+1\right]_{q}d(q)+q^{a_{m-1}+a_m}Q(q)}=\frac{c(q)+qa(q)}{d(q)+qb(q)}
\end{gathered}
\end{equation*}
Since, referring to the Stern-Brocot tree, we see that $a(1)/b(1) > c(1)/d(1)$ and $\text{ord }c(q)\leq\text{ord }a(q)$ (note that our orders are only equal if $b(1) = d(1) = 1$), this is our desired result when $a(1)/b(1)$ comes after a left branch. (Additionally, since $\left[a_m+1\right]_{1/q} = \left[a_m\right]_{1/q} + q^{-a_m}/[1]_q$, the above continued fraction expansion is correct regardless of whether $(a(1)+c(1))/(b(1)+d(1))$ comes after a left or right branch.)
\newline~
\newline
Now suppose that $a(1)/b(1)$ comes after a right branch. We have
\[
\frac{a(q)}{b(q)}=\left[\frac{a(1)}{b(1)}\right]_q=[a_0]_q+\cfrac{q^{a_0}}{[a_1]_{1/q}+\cfrac{q^{-a_1}}{\cdots+\cfrac{q^{-a_{m-1}}}{\left[a_m\right]_{q}}}}
\]
and
\[
\frac{c(q)}{d(q)}=\left[\frac{c(1)}{d(1)}\right]_q=[a_0]_q+\cfrac{q^{a_0}}{[a_1]_{1/q}+\cfrac{q^{-a_1}}{\cdots+\cfrac{q^{a_{m-2}}}{\left[a_{m-1}\right]_{1/q}}}}
\]
\newpage
\noindent In this case, the relationship between a rational number's continued fraction expansion and its placement on the Stern-Brocot tree tells us that $m$ must be even. Let
\[
\frac{P(q)}{Q(q)}=[a_0]_q+\cfrac{q^{a_0}}{[a_1]_{1/q}+\cfrac{q^{-a_1}}{\cdots+\cfrac{q^{-a_{m-3}}}{\left[a_{m-2}\right]_q}}}
\]
Using the recurrence relation governing generalized continued fractions \cite{jones} once again, we find in this case that
\[
\frac{a(q)}{b(q)}=\frac{q^{a_{m-1}}\left[a_m\right]_{q}c(q)+P(q)}{q^{a_{m-1}}\left[a_m\right]_{q}d(q)+Q(q)}
\]
and obtain
\begin{equation*}
\begin{gathered}
\left[\frac{a(1)+c(1)}{b(1)+d(1)}\right]_q=[a_0]_q+\cfrac{q^{a_0}}{[a_1]_{1/q}+\cfrac{q^{-a_1}}{\cdots+\cfrac{q^{-a_{m-1}}}{\left[a_m+1\right]_q}}}=\\
\frac{q^{a_{m-1}}\left[a_m+1\right]_{q}c(q)+P(q)}{q^{a_{m-1}}\left[a_m+1\right]_{q}d(q)+Q(q)}=\frac{a(q)+q^{a_{m-1}+a_m}c(q)}{b(q)+q^{a_{m-1}+a_m}d(q)}
\end{gathered}
\end{equation*}
Since, referring to the Stern-Brocot tree, we see that $c(1)/d(1) > a(1)/b(1)$, $\text{ord }a(q)>\text{ord }c(q)$, and \(\text{ord }a(q) - \text{ord }c(q) + 1 = a_{m-1} + a_m\), this is our desired result when $a(1)/b(1)$ comes after a right branch. (Like before, we additionally see that, since $\left[a_m+1\right]_q = \left[a_m\right]_q + q^{a_m}/[1]_{1/q}$, the above continued fraction expansion is correct regardless of whether $(a(1)+c(1))/(b(1)+d(1))$ comes after a left or right branch.) We conclude that, regardless of whether $a(1)/b(1)$ comes after a left or right branch,
\[
\frac{a(q)}{b(q)} \oplus_q \frac{c(q)}{d(q)} = \left[\frac{a(1)}{b(1)}\right]_q \oplus_q \left[\frac{c(1)}{d(1)}\right]_q = \left[\frac{a(1)}{b(1)} \oplus \frac{c(1)}{d(1)}\right]_q
\]
completing the proof.
\section*{\normalfont \textbf{Appendix B} \(\mid\) Lemma from Section 6}
We aim to prove the following lemma.
\begin{lemmasixone}
Given $\left[\frac{a}{b}\right]_q=\frac{a(q)}{b(q)}$, we have that
\[
\frac{d}{dq}a(q)\Bigr|_{q=1}=\frac{1}{2}\left(D\left(\frac{a}{b}\right)b+a^{-1}-a\right)
\]
in which $D(a/b)$ is the depth of $a/b$ on the Stern-Brocot tree and $a^{-1}$ is the inverse of $a$ in the multiplicative group \(\left(\mathbb{Z}\slash b\mathbb{Z}\right)^\times\).
\end{lemmasixone}
\begin{proof}
We set up the same four cases that we saw in the proof of the identity for the first derivative of $[a/b]_q$. In the first and second cases, $a^{-1}=b-d$, $c^{-1}=b-(n+1)d$, and $D(a/b)=D(c/d)+n-1$. In the first case, we additionally have that $(a+c)^{-1}=b$ and $D((a+c)/(b+d))=D(a/b)+1$; we see that
\begin{equation*}
\begin{gathered}
\frac{d}{dq}\left(b(q)+q^{n+1}d(q)\right)\Bigr|_{q=1}=\frac{d}{dq}b(q)\Bigr|_{q=1}+\frac{d}{dq}d(q)\Bigr|_{q=1}+(n+1)d=\\
\frac{1}{2}\left(D\left(\frac{a+c}{b+d}\right)(b+d)+(a+c)^{-1}-(a+c)\right)
\end{gathered}
\end{equation*}
as desired. In the second case, we additionally have that $(2a-c)^{-1}=b-d$ and $D((2a-c)/(2b-d))=D(a/b)+1$; we see that
\begin{equation*}
\begin{gathered}
\frac{d}{dq}\left((q+1)b(q)-q^{n}d(q)\right)\Bigr|_{q=1}=2\frac{d}{dq}b(q)\Bigr|_{q=1}+b-\frac{d}{dq}d(q)\Bigr|_{q=1}-nd=\\
\frac{1}{2}\left(D\left(\frac{2a-c}{2b-d}\right)(2b-d)+(2a-c)^{-1}-(2a-c)\right)
\end{gathered}
\end{equation*}
as desired. In the third and fourth cases, $a^{-1}=d$, $c^{-1}=3d-b$, and $D(a/b)=D(c/d)+2$. In the third case, we additionally have that $(2a-c)^{-1}=b$ and $D((2a-c)/(2b-d))=D(a/b)+1$; we see that
\begin{equation*}
\begin{gathered}
\frac{d}{dq}\left((q+1)b(q)-qd(q)\right)\Bigr|_{q=1}=2\frac{d}{dq}b(q)\Bigr|_{q=1}+b-\frac{d}{dq}d(q)\Bigr|_{q=1}-d=\\
\frac{1}{2}\left(D\left(\frac{2a-c}{2b-d}\right)(2b-d)+(2a-c)^{-1}-(2a-c)\right)
\end{gathered}
\end{equation*}
as desired. In the fourth case, we additionally have that $(a+c)^{-1}=d$ and $D((a+c)/(b+d))=D(a/b)+1$; we see that
\begin{equation*}
\begin{gathered}
\frac{d}{dq}\left(qb(q)+d(q)\right)\Bigr|_{q=1}=\frac{d}{dq}b(q)\Bigr|_{q=1}+b+\frac{d}{dq}d(q)\Bigr|_{q=1}=\\
\frac{1}{2}\left(D\left(\frac{a+c}{b+d}\right)(b+d)+(a+c)^{-1}-(a+c)\right)
\end{gathered}
\end{equation*}
as desired. This completes the proof.
\end{proof}
\newpage
\printbibliography
\end{document}